\documentclass[12pt]{article}
\usepackage{mathtools}
\usepackage{appendix}

\usepackage{amsmath,latexsym,amsfonts,enumerate}
\usepackage{amssymb,amsthm}
\usepackage{calrsfs} 
\DeclareMathAlphabet{\pazocal}{OMS}{zplm}{m}{n}
\usepackage{tikz-cd}
\usepackage{tikz}
\usepackage{comment}

\usepackage[LGR,T1]{fontenc} 
\usepackage[utf8]{inputenc} 
\usepackage{lmodern} 
\usepackage[greek,english]{babel} 
\usepackage{alphabeta} 
\usepackage{biblatex}
\usepackage{csquotes}
\usepackage[colorlinks=true,urlcolor=black,linkcolor=blue,citecolor=blue]{hyperref}
\usepackage{cleveref}
\usepackage{pgfplots}
\usepackage{tikz-3dplot}
\usepackage{faktor}
\usepackage[colorlinks=true,urlcolor=black,linkcolor=blue,citecolor=blue]{hyperref}

\usepackage[T3,T1]{fontenc}
\DeclareSymbolFont{tipa}{T3}{cmr}{m}{n}
\DeclareMathAccent{\invbreve}{\mathalpha}{tipa}{16}

\tdplotsetmaincoords{60}{115}
\pgfplotsset{compat=newest}

\newtheorem{theorem}{Theorem}[section]
\newtheorem*{theorem*}{Theorem}
\newtheorem{lemma}{Lemma}[section]
\newtheorem*{lemma*}{Lemma}
\newtheorem{corollary}{Corollary}[section]
\newtheorem{prop}{Proposition}[section]
\theoremstyle{definition}
\newtheorem{definition}{Definition}[section]
\newtheorem*{remark*}{Remark}
\newtheorem{remark}{Remark}

\newcommand{\ignore}[1]{}

\title{Classification of analytic $\text{SO}^o(p,q)$-actions on closed $(p+q-1)$-dimensional manifolds I :  $p, q \geq 3$}
\author{{\bf Spyridon Lentas} }
\date{}

\begin{document}

\maketitle

\begin{abstract}
	This paper provides a classification of analytic actions of the semi-orthogonal group $\text{SO}^\circ(p,q)$, for $p,q \geq 3$, on closed, connected $(p+q-1)$-dimensional manifolds. Adapting Uchida's construction of $\text{SO}^\circ(p,q)$ actions on $\text{S}^{p+q-1}$, we explicitly construct analytic actions of $\text{SO}^\circ(p,q)$ on $\text{S}^{p} \times \text{S}^{q-1}$ and $\text{S}^{p-1} \times \text{S}^{q}$, as well as actions on $\text{SO}^\circ(p,q) \times_P \text{S}^1$, where $P$ is a maximal parabolic subgroup of $\text{SO}^\circ(p,q)$. The main result demonstrates that any analytic $\text{SO}^\circ(p,q)$ action on a closed, connected $(p+q-1)$-dimensional manifold is covered by one of the constructed actions.

\end{abstract}

\section{Introduction}
Consider the Lie group $ \text{SO}^\circ(p,q)$, for $p,q\geq 3$. We are interested in analytic actions of $\text{SO}^\circ(p,q)$ on closed, connected manifolds. When the dimension of the manifold is $< p+q-2$, there can be no nontrivial action of $\text{SO}^\circ(p,q)$. When the dimension is $p+q-2$ there exists essentially only one action, namely the one on the homogeneous space $\text{SO}^\circ(p,q) \big/ P_{\text{null}}$, where $P_{\text{null}}\leq \text{SO}^\circ(p,q)$ is a maximal parabolic subgroup, which is isomorphic to the stabiliser of an isotropic line in the standard representation of $\text{SO}^\circ(p,q)$ on $\mathbb{R}^{p+q}$ with a scalar product of signature $(p,q)$. This paper contains results from the author's PhD dissertation, see \cite{Lentas_thesis}. We classify analytic $\text{SO}^\circ(p,q)$ actions on closed, connected manifolds of dimension $p+q-1$. More specifically, we get a classification up to analytic isomorphism for the analytic $\text{SO}^\circ(p,q)$ actions on $\text{S}^{p+q-1}$, $\text{S}^p \times \text{S}^{q-1}$, $\text{S}^{p-1} \times \text{S}^q$ and $\text{SO}^\circ(p,q) \times_{P_{\text{null}}} \text{S}^1$, where a maximal compact subgroup acts in a standard way, and then we show that an arbitrary manifold of dimension $p+q-1$ with a nontrivial $\text{SO}^\circ(p,q)$ action is $\text{SO}^\circ(p,q)$-equivariantly covered by one of these spaces. Here, $P_{\text{null}}$ acts on $\text{S}^1$ via an analytic flow, see Section \ref{axns w nullcone} for the relevant definitions. 
Theorem \ref{general thm} below gives the classification up to diffeomorphism of all manifolds admitting a nontrivial action.

\begin{theorem} \label{general thm}
Suppose $\textup{SO}^\circ(p,q)$, $p,q \geq 3$, acts analytically on a  closed, connected manifold $M$ of dimension $p+q-1$. Consider $\textup{SO}(p) \simeq \textup{SO}(p) \times \{1\} \leq \textup{SO}(p) \times \textup{SO}(q) \leq \textup{SO}^\circ(p,q)$ and $\textup{SO}(q)$ similarly.
\begin{itemize}
    \item If both $\textup{SO}(p)$ and $\textup{SO}(q)$ have a fixed point, then $M$ is covered by $\textup{S}^{p+q-1}$ 
    \item If only $\textup{SO}(p)$ has a fixed point, then $M$ is  covered by $\textup{S}^p \times \textup{S}^{q-1}$ 
    \item If only $\textup{SO}(q)$ has a fixed point, then $M$ is  covered by $\textup{S}^{p-1} \times \textup{S}^{q}$ 
    \item If neither $\textup{SO}(p)$ nor $\textup{SO}(q)$ have a fixed point, then $M$ is covered by \\
    $\textup{SO}^\circ(p,q) \times_{P_{\textup{null}}} \textup{S}^1$. If $P_{\textup{null}}=M_P A_P N_P$ is the Langlands decomposition of $P_{\textup{null}}$, then $P_{\textup{null}}$ acts on $\textup{S}^1$ by a flow via $A_P$.
\end{itemize}
\end{theorem}

 We are actually able to say more about the actions, see Theorem \ref{final result SOpq}. There are four orbit types that appear in these actions. Specifically, the orbits are diffeomorphic to the homogeneous spaces $\displaystyle \text{SO}^\circ(p,q) \big/ \text{SO}^\circ(p-1,q)$, $\text{SO}^\circ(p,q) \big/ \text{SO}^\circ(p,q-1)$, $\text{SO}^\circ(p,q) \big/ G_{\text{null}} $ or $\text{SO}^\circ(p,q) \big/ P_{\text{null}}$. Here, $G_{\text{null}} \leq \text{SO}^\circ(p,q)$ is the stabiliser of a null vector in the standard representation of $\text{SO}^\circ(p,q)$ on $\mathbb{R}^{p+q}$ and we call these orbits \textit{nullcone orbits}. Topologically these are $\mathbb{R}^q \times \text{S}^{p-1}$, $\mathbb{R}^p \times \text{S}^{q-1}$, a component of the nullcone in $\mathbb{R}^{p+q}$ with respect to a scalar form of signature $(p,q)$, and $\text{S}^{p-1}\times \text{S}^{q-1}$ respectively.

Firstly, in Section \ref{actions on Spq}, we define analytic actions of $\text{SO}^\circ(p,q)$, $p,q \geq 3$, on $\text{S}^{p} \times \text{S}^{q-1}$ where the maximal compact subgroup $ \text{SO}(p)\times \text{SO}(q)$ acts in a standard way, see the beginning of Section \ref{actions on Spq}. These actions correspond to analytic flows satisfying certain conditions, see Definition \ref{basic j1 flow}, on the fixed point set of $ \text{SO}(p-1)\times \text{SO}(q-1)$, which is a principal isotropy group with codimension 1 orbits for the action of $ \text{SO}(p)\times \text{SO}(q)$. Because of its technical nature, the proof of the analyticity of the constructed actions is moved to Appendix \ref{analyticity of actions}. Subsequently, we get a classification of those $\text{SO}^\circ(p,q)$ actions on $\text{S}^p \times \text{S}^{q-1}$, using a result of \cite{HITCHIN1991359}. The $\text{SO}^\circ(p,q)$ actions on $\text{S}^{p-1} \times \text{S}^q$ are defined similarly, while the actions on $\text{S}^{p+q-1}$ are defined in \cite{Uchida}. The actions of $\text{SO}^\circ(p,q)$ on $\text{SO}^\circ(p,q) \times_{P_{\text{null}}} \text{S}^1$ are defined in Section \ref{axns w nullcone}.

 Subsequently, we consider an analytic action of $\text{SO}^\circ(p,q)$ on a closed, connected manifold $M$ of dimension $p+q-1$. An important step towards classifying these actions is analysing the action of a maximal compact subgroup of $\text{SO}^\circ(p,q)$ and determining the orbit types that can appear, see Section \ref{SOp SOq axns} and in particular Proposition \ref{concl for SOpxSOq orbit types}. Finally, in Section \ref{SO(p,q) axns} we classify the $\text{SO}^\circ(p,q)$ actions. On the fixed point set of $\text{SO}(p-1)\times \text{SO}(q-1)$, which is a finite union of circles, we get an analytic flow, like above.
If there are no nullcone orbits, then this flow helps us define an $\text{SO}^\circ(p,q)$ action on $\text{S}^{p+q-1}$, $\text{S}^p \times \text{S}^{q-1}$ or $\text{S}^{p-1} \times \text{S}^q$ depending on the existence of fixed points of $\text{SO}(p)$ and $\text{SO}(q)$, and get an $\text{SO}^\circ(p,q)$-equivariant covering map to $M$. In the presence of a nullcone orbit, which is equivalent to neither $\text{SO}(p)$ nor $\text{SO}(q)$ having fixed points, $M$ is $\text{SO}^\circ(p,q)$-equivariantly covered by $\text{SO}^\circ(p,q) \times_{P_{\text{null}}} \text{S}^1$.

In a second paper, which is forthcoming, we will deal with the classification of analytic $\text{SO}^o(p,q)$ actions in the case $p=1,2$.

\subsection{Background and Motivation}

In \cite{Uchida}, Uchida studied smooth actions of $\text{SO}^\circ(p,q)$ on $\text{S}^{p+q-1}$ with the extra assumption that the action of a maximal compact subgroup on $\text{S}^{p+q-1}$ is the standard orthogonal one. He showed that such actions are in one to one correspondence with pairs $(\Phi, f)$ where $\Phi$ is a smooth flow on $\text{S}^1$ and $f : \text{S}^1 \rightarrow \mathbb{R}\text{P}^1$ is smooth, and $\Phi$ and $f$ satisfy certain conditions. We will briefly recall Uchida's construction in Section \ref{actions on Spq}. Here, $\text{S}^1$ is diffeomorphic to the fixed point set of $H$, which is a principal isotropy group with codimension 1 orbits for the action of the maximal compact subgroup. The function $f$ encodes the isotropy algebra at a point in $\text{S}^1$ for the given action of $\text{SO}^\circ(p,q)$, see Lemma \ref{lemma Uch} and Remark \ref{rmk definition of f}. Note that Uchida's result gives a correspondence between smooth $\text{SO}^\circ(p,q)$ actions and pairs $(\Phi, f)$, but only classifies them up to homeomorphism. 

 Such a pair $(\Phi, f)$ was first introduced by Asoh in \cite{AsohSL2C} in order to study actions of $\text{SL}_2(\mathbb{C})$ on $\text{S}^3$. It was then modified by Uchida in \cite{Uchida} to study the actions mentioned above and modified even further by Mukoyama in \cite{MukoyamaSp2r} to study smooth $\text{Sp}(2,\mathbb{R})$ actions on $\text{S}^4$. Uchida in \cite{Uchida_SOpqII} used Mukoyama's approach to study smooth $\text{SO}^\circ(p,2)$ actions on $\text{S}^{p+1}$, as well as $\text{Sp}(p,q)$ actions on $\text{S}^{4p+4q-1}$ in \cite{UchidaSppq}. Mukoyama also studied smooth $\text{SU}(p,q)$ actions on $\text{S}^{2p+2q-1}$ and $\mathbb{R} \text{P}^{p+q-1}$ in \cite{MukoyamaSUpq}. On the other hand, assuming analyticity, Schneider classified $\text{SL}_2(\mathbb{R})$ actions on 2-dimensional compact manifolds in \cite{SchneiderSL2axns}. We note that $\text{SL}_2(\mathbb{R})$, $\text{SL}_2(\mathbb{C})$ and $\text{Sp}(2,\mathbb{R})$ are locally isomorphic to $\text{SO}^\circ(1,2)$, $\text{SO}^\circ(1,3)$ and $\text{SO}^\circ(2,3)$ respectively.

 In a different direction, Uchida classified analytic $\text{SL}_n(\mathbb{R})$ actions on $\text{S}^n$ in \cite{Uchida_clssf_rlanalytic_SLn}. Later, both the analytic and the smooth $\text{SL}_n(\mathbb{R})$ actions on $n$-dimensional closed manifolds were classified by Fisher and Melnick in \cite{FisherMelnick}. Son, in \cite{Son_Reanalytic_SLn_n<m<2n-2}, extended the result in the analytic case to dimensions $m$, where $n\leq m \leq 2n-3$. She also classified the smooth actions in these dimensions when the action if fixed point free.

\subsection{Notation}
Throughout the text we use the following notation

\begin{itemize}
    \item[-] If $A$ is a Lie group, $A^\circ$ denotes its identity component
    \item[-] $ \displaystyle G := \text{SO}^\circ(p,q)= \left\{ X \in \text{SL}_{p+q}(\mathbb{R}) : X \, I_{p,q} \, X^T = I_{p,q} \right\}$, $p\geq3$,\\ where $I_{p,q} = \begin{bmatrix}
    -I_p & \\
    & I_q
\end{bmatrix}$
where $I_p$ and $I_q$ are the identity $p \times p$ and $q \times q$ matrices respectively.
    \item[-] $ \displaystyle K := \left\{ \begin{bmatrix}
    A & \\ & B 
\end{bmatrix} \in G : A \in \text{SO}(p) , B \in \text{SO}(q) \right\} \leq G   $, $K \simeq \text{SO}(p) \times \text{SO}(q)$.
    \item[-] $ \displaystyle H := \left\{ \begin{bmatrix}
    1 & && \\ 
    & \Tilde{A} && \\ 
    && 1 & \\
    &&& \Tilde{B}
\end{bmatrix} \in K : \Tilde{A} \in \text{SO}(p-1) , \Tilde{B} \in \text{SO}(q-1) \right\} \leq K$,\\
$H \simeq \text{SO}(p-1) \times \text{SO}(q-1)$
    \item[-] $\mathcal{M}(p,q) = \left\{ \begin{bmatrix}
    \text{cosh}(\theta) & & \text{sinh}(\theta) & \\
    & I_{p-1} && \\
    \text{sinh}(\theta) & & \text{cosh}(\theta) & \\
    &&& I_{q-1}
\end{bmatrix} \right\} \leq G$
    \item[-] $j_1 = \begin{bmatrix}
        -1&& \\ & -1 & \\ && I_{p+q-2}
    \end{bmatrix} \in \text{SO}(p) \times \{ I_q \} \leq K$
    \item[-] $j_2 = \begin{bmatrix}
        I_p &&& \\ & -1 && \\ &&-1 & \\ &&& I_{q-2}
    \end{bmatrix} \in \{ I_p \} \times \text{SO}(q) \leq K$
    \item[-] If $G$ is a group and $X$ is a space on which $G$ acts, then for $g \in G$ and $x \in X$ we will denote the action of $g$ on $x$ by $g \star x$. An exception to this will be when a matrix $A$ acts on a vector $v$, where we will just write $Av$.
    \item[-] If $A$ is a Lie group, $B \leq A$ a Lie subgroup and $X$ is a space on which $B$ acts, then $\displaystyle A \times_B X$ is the quotient space of $A \times X$ module the equivalence relation $\displaystyle (a,x) \sim (ab, b^{-1}x)$. $A$ acts on $A \times_B X$ by multiplication on the left: $\displaystyle \Tilde{a} \star (a,x) = (\Tilde{a}a,x)$ 
    
\end{itemize}

\subsection*{Acknowledgments}
The author wishes to express his deep gratitude to his advisor, Prof. Karin Melnick, for her invaluable guidance and unwavering support throughout the completion of this project. The author is also thankful to Prof. Bill Goldman for several helpful discussions.




\section{Construction of $\text{SO}^\circ(p,q)$ actions on $\text{S}^p\times \text{S}^{q-1}$ and $\text{S}^{p-1}\times \text{S}^q$} \label{actions on Spq}

Assume $p,q \geq 3$ and consider $G:=\text{SO}^\circ(p,q)$. We see $\text{S}^p$ as a subset of $\mathbb{R}^{p+1}$ with basis $\{ e_1, \cdots , e_{p+1} \}$, and S$^{q-1}$ as a subset of $\mathbb{R}^q$ with basis $\{ \epsilon_1, \cdots , \epsilon_q \}$.  We consider the following action of $K := \text{SO}(p) \times \text{SO}(q)$, which we call the \textit{standard} action of $K$ on $\text{S}^p \times \text{S}^{q-1}$: \\
First, we embed $\text{SO}(p)$ in $\text{SO}(p+1)$ by 
\begin{equation} \label{SOp embed in SOp+1}
    \kappa \mapsto \tilde{\kappa} := 
    \begin{bmatrix}
        \kappa & \\ & 1
    \end{bmatrix}
\end{equation}
Then, $K$ acts on $\text{S}^p \times \text{S}^{q-1}$ by matrix-vector multiplication in the obvious way: $\text{SO}(p)$ acts on the first factor via its embedding in $\text{SO}(p+1)$ and $\text{SO}(q)$ acts on the second factor. Evidently, this action is analytic. In this section we will construct $G$ actions on $\text{S}^p \times \text{S}^{q-1}$ that extend the standard $K$ action. Note that for these actions $\text{SO}(p) \times \{ I_q\} \leq K$ has fixed points, while $\{I_p \} \times \text{SO}(q) \leq K$ does not.

Let $\mathcal{F}$ be the fixed point set of the subgroup $H:=\text{SO}(p-1) \times\text{SO}(q-1)$. Then, $\mathcal{F}$ comprises the circles $\big\{ ( \alpha e_1 + \beta e_{p+1} , \epsilon_1 ) : \alpha^2 + \beta^2 =1 \big\}$ $\big\{ ( \alpha e_1 + \beta e_{p+1} , - \epsilon_1 ) : \alpha^2 + \beta^2 =1 \big\}$. Set
\[ \mathcal{S} =  \big\{ ( \alpha e_1 + \beta e_{p+1} , \epsilon_1 ) : \alpha^2 + \beta^2 =1 \big\} \]
Recall, $j_1 = \begin{bmatrix}
        -1&& \\ & -1 & \\ && I_{p+q-2}
    \end{bmatrix} \in K$, which acts on $\mathcal{S}$ by $j_1 \star ( \alpha e_1 + \beta e_{p+1} , \epsilon_1 ) \mapsto ( -\alpha e_1 + \beta e_{p+1} , \epsilon_1 )$. Note that, the other connected component of $\mathcal{F}$ is equal to $j_2 \star \mathcal{S}$. Recall that by ``$\star"$ we denote the action of a group element, see the Notation section in the Introduction.

The main ingredient for our construction will be a special type of flows on $\text{S}^1$. Here, we see $\text{S}_1$ as the unit circle in $\mathbb{R}^2$, namely $\text{S}^1 = \{ \alpha \, e_1 + \beta \, e_2 : \alpha^2+\beta^2=1 \}$.
\begin{definition} \label{basic j1 flow}
    Assume $\Phi_\theta$ is a nontrivial analytic flow on $\textup{S}^1$ and let  $\textup{J}_1$ be the reflection with respect to the $y$-axis, namely $\textup{J}_1 ( \alpha \, e_1 + \beta \, e_2) = - \alpha \, e_1 + \beta \, e_2$. We will say that $\Phi_\theta$ is a \underline{basic $\textup{J}_1$-flow} if:
    \begin{itemize}
        \item It has exactly two fixed points on $\textup{S}^1$, none of which are $\pm e_2$.
        \item $\textup{J}_1 \Phi_\theta(z) = \Phi_{-\theta} \left( \textup{J}_1 z \right)$, for $\theta \in \mathbb{R}$ and $z \in \textup{S}^1$.
        \item The Jacobian of $\Phi_\theta$ is $-\frac{2}{n}$, respectively $\frac{2}{n}$, at the attracting, respectively repelling, fixed point, where $n \in \mathbb{N}$.
    \end{itemize}
\end{definition}
 Note that then, if $z_1, z_2$ are the fixed points of $\Phi_\theta$, we have $\textup{J}_1 \star z_1 = z_2$. The actions we will construct in this section will correspond to basic $\text{J}_1$-flows. Note that, one of the fixed points of $\Phi_\theta$ must be attracting and the other repelling.

\begin{remark} \label{remark classification of basic j1 fl}
Suppose $\Phi_\theta$ is a basic $\textup{J}_1$-flow on $\text{S}^1$ and let $\partial_{\text{S}^1}$ be a basic vector field of $\text{S}^1$. Then, $\Phi_\theta$ generates an analytic vector field $X = g \cdot \partial_{\text{S}^1}$. In \cite{HITCHIN1991359}, Hitchin classified analytic vector fields based on a number of local and global invariants. For a basic $\textup{J}_1$-flow, since we have two fixed points one of which is necessarily attracting and the other necessarily repelling, the only invariants that matter are the Jacobian at the fixed points, which is $\pm 2/n$ and the \textit{global invariant} $\mu$ which can be thought of as $\mu = \int_{\textup{S}^1  } \frac{1}{g}$. Of course, in our case the right hand side does not make sense, but it can be defined using complex integration methods around the zeros of $g$, see \cite{HITCHIN1991359}.
\end{remark}
Now, let 
\begin{equation} \label{rho: S1 -> mathcal S}
    \rho_0 : \text{S}^1 \rightarrow \mathcal{S}
\end{equation}
be the analytic isomorphism defined by $\alpha \, e_1 + \beta \, e_2 \mapsto \left( \alpha \, e_1 + \beta \, e_{p+1}, \epsilon_1 \right)$.
\begin{definition} \label{induced basic j1 flow}
    Suppose $\Phi_\theta$ is a basic $\textup{J}_1$-flow on $\textup{S}^1$. Via the isomorphism $\rho_0$, see (\ref{rho: S1 -> mathcal S}), $\Phi_\theta$ induces an analytic flow $\Phi_\theta'$ on $\mathcal{S}$. We will call analytic flows on $\mathcal{S}$ arising this way, \underline{induced basic $j_1$-flows}. 
\end{definition}
The following is immediate:
\begin{lemma} \label{induced basic j1 flows properties}
Suppose $\Phi_\theta$ is a nontrivial analytic flow on $\mathcal{S}$. Then, $\Phi_\theta$ is an induced basic $j_1$-flow on $\mathcal{S}$ if and only if it satisfies the following relations:
\begin{itemize}
        \item It has exactly two fixed points on $\mathcal{S}$, none of which are $\left( e_{p+1}, \epsilon_1 \right)$ or $\left( -e_{p+1}, \epsilon_1 \right)$.
        \item $j_1 \Phi_\theta(z) = \Phi_{-\theta} \left( j_1 z \right)$, for $\theta \in \mathbb{R}$ and $z \in \mathcal{S}$.
        \item The Jacobian of $\Phi_\theta$ is $-\frac{2}{n}$, respectively $\frac{2}{n}$, at the attracting, respectively repelling, fixed point, where $n \in \mathbb{N}$.
\end{itemize}
\end{lemma}
Furthermore, the following function, defined in terms of an induced basic $j_1$-flow, will be useful to us later:
\begin{remark} \label{f_phi def remark}
Let $\Phi_\theta$ be an induced basic $j_1$-flow on $\mathcal{S}$. Via $\Phi_\theta$ we can define a function, $f_\Phi$ on $\mathcal{S}$ in the following way:
\begin{itemize}
    \item $f_\Phi(\pm e_{p+1},\epsilon_1) := 0$
    \item $f_\Phi(z) : = \text{tanh}(\theta)$, for a point  $z \in \mathcal{S}$ of the form $z = \Phi_\theta \left( (\pm e_{p+1},\epsilon_1) \right)$
    \item $f_\Phi(z_1): = 1$ and $f_\Phi(z_2):=-1$
\end{itemize}
\end{remark}

 Analogously, we can define a standard action of $K$ on $\text{S}^{p-1} \times \text{S}^q$, \textit{basic $\textup{J}_2$-flows} on $\text{S}^1$, and \textit{induced basic $j_2$-flows} on the connected component of the fixed point set of $H$, $\mathcal{S}_2 = \left\{ \left( e_1 , \alpha \epsilon_1 + \beta \epsilon_{q+1} \right) : \alpha^2 + \beta^2 =1 \right\}$. The analogous construction of the one we will describe below will give $G$ actions on $\text{S}^{p-1} \times \text{S}^q$ which extend the standard action of $K$. Note that for these actions, $\{I_p \} \times \text{SO}(q)$ has fixed points, but $\text{SO}(p) \times \{ I_q\}$ does not. In the rest of this section we will deal with the $\text{S}^p \times \text{S}^{q-1}$ case. The statements and proofs for the $\text{S}^{p-1} \times \text{S}^q$ are completely analogous.

\subsection{The Basic Construction} \label{basic constr}

In this section we construct a $G$ action on $\text{S}^{p} \times \text{S}^{q-1}$ given a basic $\textup{J}_1$-flow on $\text{S}^1$. This construction is adapted from the construction in \cite{Uchida}, see Section \ref{Uchida constr section}. Let $m(\theta) := \begin{bmatrix}
    \text{cosh}(\theta) & & \text{sinh}(\theta) & \\
    & I_{p-1} && \\
    \text{sinh}(\theta) & & \text{cosh}(\theta) & \\
    &&& I_{q-1}
\end{bmatrix}$ and set 
\[ \mathcal{M}(p,q) = \left\{ m(\theta) \, : \, \theta \in \mathbb{R} \right\} \]
which is a one-parameter subgroup of $G$. Moreover, consider the isotropy group of the point $a \, e_1 + b \, e_{p+1} \in \mathbb{R}^{p+q}$, for $[a:b] \in \mathbb{R}\text{P}^1$, in the standard representation of $G$ on $\mathbb{R}^{p+q}$, which we denote by $H_{[a:b]} \leq G$. The following equality is established in \cite{Uchida}:
\begin{equation} \label{G decomp}
    G = K \, \mathcal{M}(p,q) \, \left( H_{[a:b]} \right)^\circ
\end{equation}
for any $[a:b] \in \mathbb{R}\text{P}^1$.

Suppose $\widetilde{\Phi}_\theta$ is a basic $\textup{J}_1$-flow on $\text{S}^1$, see Definition \ref{basic j1 flow}. Let $\Phi_\theta$ be the induced basic $j_1$-flow on $\mathcal{S}$, see Definition \ref{induced basic j1 flow}. Recall that $\mathcal{S} =  \big\{ ( \alpha e_1 + \beta e_{p+1} , \epsilon_1 ) : \alpha^2 + \beta^2 =1 \big\}$ is a connected component of the fixed point set, $\mathcal{F}$, of $H$ in the standard action of $K$ on $\text{S}^{p} \times \text{S}^{q-1}$. Recall also the function $f_\Phi$, see Remark \ref{f_phi def remark}.  For $z \in \mathcal{S}$, let \
\begin{equation} \label{U(z) eq definition}
 U(z) = (H_{[f_\Phi(z):1]})^\circ
\end{equation}
Then, equation (\ref{G decomp}) can be written as
\begin{equation} \label{eq:1}
    G = K \, \mathcal{M}(p,q) \, U(z)
\end{equation}
for any $z \in \mathcal{S}$. Now, let $g \in G$ and $(v,w) \in \text{S}^p \times \text{S}^{q-1}$. There exists $z \in \mathcal{S}$ and $k_0 = (\kappa_1, \kappa_2) \in K$ such that $k_0 \star z = (v,w)$. Write $gk_0 = k \, m(\theta) \, u_z$ , according to (\ref{eq:1}). Then, we define
\begin{equation} \label{axn def}
    g \star (v,w) : = k \star \Phi_{\theta}(z) 
\end{equation}
We prove that this action is well-defined in Section \ref{subsection: well-def of axns} and we show that it is analytic in Appendix \ref{analyticity of actions}. We will refer to $G$ actions defined this way as \textit{actions from the basic construction}. Note that $\Phi_\theta (z)$ is of the form $\big( \alpha_{\Phi_\theta (z)} e_1 + \beta_{\Phi_\theta (z)} e_{p+1} , \epsilon_1 \big)$, for some $\alpha_{\Phi_\theta (z)}, \beta_{\Phi_\theta (z) } \in \mathbb{R}$ such that $\alpha_{\Phi_\theta (z) }^2 + \beta_{\Phi_\theta (z) }^2 = 1$.

\begin{remark} Although the basic construction is described on $\text{S}^p \times \text{S}^{q-1}$ in terms of basic $\text{J}_1$-flows, there is an obvious analogue on $\text{S}^{p-1} \times \text{S}^q$ in terms of basic $\text{J}_2$-flows. We will will also refer to those actions as actions from the basic construction when it is clear we are talking about actions of $G$ on $\text{S}^{p-1} \times \text{S}^q$.
\end{remark}

\begin{theorem} \label{basic axns thm}
Two actions from the basic construction are analytically isomorphic if and only if the corresponding basic $\textup{J}_1$-flows are analytically isomorphic. Moreover, the basic $\textup{J}_1$-flows are classified up to analytic isomorphism by the Jacobian at the fixed points, which is $\pm 2/n$, for $n \in \mathbb{N}$, and the global invariant $\mu$.
\end{theorem}

 The first part of Theorem \ref{basic axns thm} is proved in Section \ref{isom axns}. For the second part see Remark \ref{remark classification of basic j1 fl}. Of course, there is an analogous theorem for the analytic $\textup{SO}^\circ(p,q)$ actions on $\text{S}^{p-1} \times \text{S}^q$:

\begin{theorem}
Two $\textup{SO}^\circ(p,q)$ actions on $\textup{S}^{p-1} \times \textup{S}^q$ from the basic construction are analytically isomorphic if and only if the corresponding basic $\textup{J}_2$-flows are analytically isomorphic. The basic $\textup{J}_2$-flows are classified up to analytic isomorphism by the Jacobian at the fixed points and the global invariant $\mu$.
\end{theorem}

\subsection{Properties of $\textup{SO}^\circ(p,q)$ actions on $\textup{S}^p \times \textup{S}^{q-1}$} \label{properties of SOpq axns on Sp x Sq-1}

 Suppose we have an analytic action of $G$ on $\textup{S}^p \times \textup{S}^{q-1}$ that extends the standard $K$ action. We are going to extract some important data for such actions, which will be useful later. The results in this section also show that these kind of actions are action from the basic construction, see in particular Lemma \ref{relations <-> basic flow}. Recall that $H_{[a:b]}$ denotes the isotropy subgroup of the point $a \, e_1 + b \, e_{p+1} \in \mathbb{R}^{p+q}$ in the standard representation of $G$ on $\mathbb{R}^{p+q}$ and let $\mathfrak{h}_{[a:b]}$ denote its Lie algebra.
\begin{lemma}{\cite[Lemma 1.7]{Uchida}} \label{lemma Uch}
    Suppose $p,q \geq 3$. Let $\mathfrak{a}$ be a proper subalgebra of $\mathfrak{so}(p,q)$ which contains $\mathfrak{h} \simeq \mathfrak{so}(p-1) \oplus \mathfrak{so}(q-1)$. If 
    \[\textup{dim} \, \mathfrak{so}(p,q) - \textup{ dim} \, \mathfrak{a} \leq p +q-1\]
    then $\mathfrak{a} = \mathfrak{h}_{[a:b]}$ for some $(a,b) \neq (0,0)$ or $\mathfrak{a} = \mathfrak{h}_{[1: \epsilon]} \oplus \theta^1$ for $\epsilon = \pm 1$, where the one-dimensional space $\theta^1$ is generated by the matrix $E_{1,p+1} + E_{p+1,1}$
\end{lemma}
\begin{remark} \label{rmk definition of f}
Recall that $\mathcal{F}$ is the fixed point set of $H$ and that $\mathcal{F} = \mathcal{S} \bigcup j_2 \star \mathcal{S}$. The above lemma allows us to define a function $\tilde{f}: \mathcal{F}\rightarrow \mathbb{R}\text{P}^1$ in the following way: for $z \in \mathcal{F}$, set $\tilde{f}(z):= [a_z:b_z]$, where $[a_z:b_z] \in \mathbb{R}\text{P}^1$ is the unique point such that $\mathfrak{h}_{[a_z:b_z]} \leq \mathfrak{g}_z$, where $\mathfrak{g}_z$ is the Lie isotropy algebra at $z$. The function $\tilde{f}$ is analytic, see \cite[p. 778]{Uchida}. This way of defining a function $\tilde{f}$ on the fixed point set of $H$ will be used numerous times.
    \end{remark}
 For the actions that we consider, since there is no point on $\mathcal{F}$ fixed by $\text{SO}(q)$, the second coordinate of $\tilde{f}$ would always be non-zero. So, we can define an analytic function $f$ by $f(z)= \frac{a}{b} \text{ if } \tilde{f}(z) = [a:b]$. Additionally, it is easy to see that $\mathcal{M}(p,q)$ normalises $H$ and hence, it preserves $\mathcal{F}$. The action of $\mathcal{M}(p,q)$ gives an analytic flow $\Phi_\theta$ on $\mathcal{F}$. Abusing the notation, we will write $\text{SO}(p)$ for $\text{SO}(p) \times \{ I_q\} \leq K$ and $\text{SO}(q)$ for $\{I_p \} \times \text{SO}(q) \leq K$. Recall $j_1 = \begin{bmatrix}
    -I_2 & \\ & I_{p+q-2}
\end{bmatrix} \in \text{SO}(p)$ and $j_2 = \begin{bmatrix}
    I_{p+q-2} & \\ & -I_2
\end{bmatrix} \in  \text{SO}(q)$. 

\begin{remark} \label{(A1)-(A4) props remark}
By straightforward matrix multiplications, we see that $\Phi_\theta$ and $f$ satisfy the following properties, for $z \in \mathcal{S}$ and $\theta \in \mathbb{R}$:
\begin{itemize}
    \item[(A1)] $ \displaystyle j_i \star \Phi_\theta (z) = \Phi_{- \theta} (j_i \star z)$ \qquad $(i=1,2)$
    \item[(A2)] $ f(j_i \star z) = -f(z)$ \qquad $(i=1,2)$
    \item[(A3)] $\displaystyle f\big( \Phi_\theta (z) \big) = \frac{f(z) + \text{tanh}(\theta)}{1+f(z) \text{tanh}(\theta)}$ 
    \item[(A4)] $\displaystyle f(z) = 0 \Leftrightarrow z = (\pm e_{p+1}, \epsilon_1)$ or $(\pm e_{p+1}, - \epsilon_1)$
\end{itemize}
Note that (A4) can be written alternatively as
\begin{itemize}
    \item[$\text{(A4)}^\prime$] $\displaystyle f(z) = 0 \Leftrightarrow z$ is fixed by $\text{SO}(p)$
\end{itemize}   
\end{remark}

\begin{lemma} \label{relations <-> basic flow}
    \begin{itemize}
        \item[(i)] Suppose $\Phi_\theta$ is an induced basic $j_1$-flow on $\mathcal{S}$, see Definition \ref{induced basic j1 flow}. Then, $f_\Phi$, see Remark \ref{f_phi def remark}, is analytic, and $\Phi_\theta$ and $f_\Phi$ can be extended to $\mathcal{F}$ so that they satisfy relations \textup{(A1)-(A4)}.
        \item[(ii)] Suppose $\Phi_\theta$ and $f$ are an analytic flow and function respectively on $\mathcal{F}$ such that they satisfy relations \textup{(A1)-(A4)}. Then, $\Phi_\theta'=\displaystyle \Phi_\theta \big|_\mathcal{S}$ is an induced basic $j_1$-flow on $\mathcal{S}$ and $\displaystyle f_{\Phi_{\theta}'} = f$.
    \end{itemize}
\end{lemma}
\begin{proof}
\begin{itemize}
    \item[(i)] 
Suppose $\Phi_\theta$ is an induced basic $j_1$-flow on $\mathcal{S}$. Firstly, we prove that $f_\Phi$ is analytic on $\mathcal{S}$. Observe that, $f_\Phi$ is evidently analytic on the orbits of $(\pm e_{p+1}, \epsilon_1)$. To prove the analyticity at a fixed point of $\Phi_\theta$, say $z_1$, we can use Poincar\'e's theorem, see \cite[Section 22]{Arnold_geom_methds_in_ODEs}, to linearise the flow around $z_1$. The theorem implies that there exists a change of coordinates around $z_1$ such that, after using a chart centered at $z_1$, the flow has the form $ \displaystyle \Phi_\theta(x) = x \, e^{tJ_\Phi(z_1)}$, where $J_\Phi$ is the Jacobian of $\Phi_\theta$ at $z_1$, which is nonzero. Then, let $y$ be a point in the domain in which the above form of $\Phi_\theta$ is valid. Note that, by (A3), $f_\Phi$ takes the form $ \displaystyle f_\Phi \left( \Phi_\theta(y) \right) = \frac{f(y) + \text{tanh}(\theta)}{1+f(y) \text{tanh}(\theta)}$.  
Let $x = \Phi_\theta(y)$. Since $\Phi_\theta(y) = y \, e^{tJ_\Phi(z_1)}$, we can solve for $\theta$ and get $ \displaystyle \theta = \frac{1}{J_\Phi(z_1)} \, \ln \left(\frac{x}{y}\right)$. Then, a straightforward calculation shows that $ \displaystyle \text{tanh}(\theta) = \frac{1-\left(\frac{x}{y}\right)^{2/J_\Phi(z_1)}}{1+\left(\frac{x}{y}\right)^{2/J_\Phi(z_1)}} $. But the right hand side of this equality is an analytic function in a neighbourhood of $x=0$. Hence, $f_\Phi$ is analytic at $z_1$. Similarly, it is shown that $f_\Phi$ is analytic at the other fixed point of $\Phi_\theta$ as well.  Subsequently, we extend $\Phi_\theta$ to all of $\mathcal{F} = \mathcal{S} \bigcup j_2 \star \mathcal{S}$, by defining for $z \in j_2 \star \mathcal{S}$ and $\theta \in \mathbb{R}$, $\Phi_\theta (z) : = j_2 \star \left( \Phi_{-\theta} \left( j_2 \star z \right) \right)$. Similarly, we extend $f_\Phi$ on $j_2\star \mathcal{S}$ by defining $f_\Phi (z) : = - f_\Phi (j_2 \star z)$. Clearly, $\Phi_\theta$ and $f_\Phi$ satisfy the relations (A1)-(A4).  Evidently, $\Phi_\theta$ and $f_\Phi$ are analytic on $\mathcal{F}$. 
\item[(ii)] Suppose $\Phi_\theta$ and $f$ are an analytic flow and function such that they satisfy relations (A1)-(A4). The fact that the fixed points of $\displaystyle \Phi_\theta'$ are not $(\pm e_{p+1}, \epsilon_1)$ and the relation $\displaystyle j_1 \star \Phi_\theta ' (z) = \Phi_{-\theta} ' (j_1 \star z)$ are immediate. The flow $\Phi_\theta'$ generates a vector field $X$ on $\mathcal{S}$. If $\partial_\mathcal{S}$ is a basic vector field of $\mathcal{S}$, the we can write $X = g \cdot \partial_\mathcal{S}$, where $g$ is an analytic function on $\mathcal{S}$. Then, property (A3) for $z=(e_{p+1}, \epsilon_1)$ shows that $f$ satisfies the following differential equation
\begin{equation} \label{ODE f and g}
 g \, \partial_\mathcal{S} \left( f \right) = 1 - f^2
 \end{equation}
Note that this equation is valid at the fixed points of $\Phi_\theta'$ also, since at those points $g=0$ and $f = \pm 1$. The last equality follows from (A3) by taking $\theta \rightarrow \pm \infty$. Let $z_1$ be one of the fixed points of $\Phi_\theta$ for which $f(z_1)=1$. We can pick a chart centered at $z_1$, where (\ref{ODE f and g}) takes the form
\begin{equation} \label{ODE f and g on interval}
    g \cdot f' = 1- f^2
\end{equation}
The functions in (\ref{ODE f and g on interval}) are to be understood as composed with the chart, but we still use the same letters. Since $f$ is an analytic function on $\mathcal{S}$ and since it is not constantly equal to 1, it cannot be constantly equal to 1 around $z_1$. That means that the order of vanishing of $1-f$ in (\ref{ODE f and g on interval}) at $0$ cannot be infinite. By order of vanishing we mean the smallest number $n \in \mathbb{N}$ such that $(1-f)^{(n)}(0)= f^{(n)}(0) \neq 0$, but $(1-f)^{(k)}(0) = 0$ for $0 \leq k < n$, where $(1-f)^{(0)} = 1-f$. But then, differentiating equation (\ref{ODE f and g on interval}) $(n+1)$-times and evaluating at $0$, we get that $g'(0) = - \frac{2}{n}$
Of course, $g'(0)$ is equal to the Jacobian of $\Phi_\theta'$ at $z_1$. By (A1), we get that $g \circ j_1 = g$, hence we see that at the other fixed point of $\Phi_\theta'$, which is $j_1(z_1)$, we have that the Jacobian of $\Phi_\theta'$ is $J_\Phi\left( j_1(z_1) \right) = \frac{2}{n}$. Then, by Lemma \ref{induced basic j1 flows properties}, $\Phi_\theta'$ is an induced basic $j_1$-flow.
 \end{itemize}
\end{proof}

\begin{remark}
    In the following, we will drop the subscript $``\Phi"$ from $f_\Phi$ and we will simply write $f$.
\end{remark}

\subsection{The Uchida construction} \label{Uchida constr section}
 As it was mentioned in the Introduction, in \cite{Uchida}, Uchida studied smooth actions of $G$ on $\text{S}^{p+q-1}$ extending the standard orthogonal action of $K$. 
\begin{remark} \label{Uchida conditions}
Uchida showed that such actions are in one to one correspondence with pairs $(\Phi, \tilde{f})$ where $\Phi$ is a smooth flow on $\text{S}^1$ and $\tilde{f} : \text{S}^1 \rightarrow \mathbb{R}\text{P}^1$ is smooth, and $\phi$ and $\tilde{f}$ satisfy:
\begin{itemize}
    \item[(i)] $j_i \star \Phi_\theta (z) = \Phi_{- \theta} (j_i \star z)$ \qquad $(i=1,2)$
    \item[(ii)] $\tilde{f}(z) = [a:b] \Rightarrow \tilde{f}(j_i \star z) = [a : -b]$ \qquad $(i=1,2)$
    \item[(iii)] $\tilde{f}(z) = [a:b] \Rightarrow \tilde{f} \left( \Phi_\theta (z) \right) = \left[ a \, \text{cosh}(\theta) + b \, \text{sinh}(\theta) : a \, \text{sinh}(\theta) + b \, \text{cosh}(\theta) \right] $ 
    \item[(iv)] $\tilde{f}(z) = [0:1] \Leftrightarrow z = e_{p+1} \quad \text{and} \quad \tilde{f}(z) = [1:0] \Leftrightarrow z = e_1$
\end{itemize}
where, $j_1$ and $j_2$ act on $\text{S}^1$ as the reflections with respect to the $y$-axis and $x$-axis respectively, see \cite[Theorem]{Uchida}. We will refer to the above conditions as \textit{Uchida conditions}.
\end{remark}

The correspondence roughly goes as follows. Given a $G$ action on $\text{S}^{p+q-1}$, let $\mathcal{F}$ be the fixed point set of $H$. Then, $\mathcal{F}$ is diffeomorphic to $\text{S}^1$.  The pair $(\Phi, \tilde{f})$ is obtained via $\mathcal{M}(p,q)$ and Lemma \ref{lemma Uch} respectively. For the reverse direction, suppose a pair $(\Phi, \tilde{f})$ satisfying the Uchida conditions is given. We can assume that $(\Phi, \tilde{f})$ are a flow and a function on $\mathcal{F}$, since $\mathcal{F}$ is diffeomorphic to $\text{S}^1$. For $z \in \mathcal{F}$ with $\tilde{f}(z) = [a:b]$, let $ U'(z) = (H_{[a:b]})^\circ$. Then, by equation (\ref{G decomp}), $G = K \, \mathcal{M}(p,q) \, U'(z)$ for any $z \in \mathcal{F}$. Therefore, an action of $G$ on $\text{S}^{p+q-1}$ such that $K$ acts in the standard way is defined in the following way: Let $v \in \text{S}^{p+q-1}$ and $g \in G$. There exist $z \in \text{S}^1$ and $k \in K$ such that $k \star z = v$. Using (\ref{eq:1}), write $ gk = k_1 \, m(\theta) \, u$. Then, the action is defined as $ g \star v := k_1 \star \Phi_\theta (z) $, where on the right hand side the action of $k_1$ is the standard orthogonal one.

Let $\mathcal{S}_3 = \{ \alpha \, e_1 + \beta \, e_{p+1} : \alpha^2+\beta^2=1 \} \subseteq \text{S}^{p+q-1}$. Then $\mathcal{S}_3=\mathcal{F}$ for $G$ actions on $\text{S}^{p+q-1}$ extending the orthogonal $K$ action. 

\begin{definition}
    Assume $\Phi_\theta$ is a nontrivial analytic flow on $\textup{S}^1$. We will say that $\Phi_\theta$ is a \underline{basic $(\textup{J}_1,\textup{J}_2)$-flow} if:
    \begin{itemize}
        \item It has exactly 4 fixed points on $\text{S}^1$, none of which are the points $\pm e_1$ or $\pm e_{2}$.
        \item $\textup{J}_i \Phi_\theta (z) = \Phi_{- \theta} (\text{J}_i z)$, $i=1,2$, where $\textup{J}_1$ is the reflection with respect to the $y$-axis and $\textup{J}_2$ the reflection with respect to the $x$-axis.
        \item The Jacobian of $\Phi_\theta$ at the fixed points is $\pm 2/n$, where $n \in \mathbb{N}$.
    \end{itemize}
\end{definition}
 Note that the second condition in the above definition implies that if $z_1$ is a fixed point of $\Phi_\theta$, then the rest of the fixed points are $\text{J}_i*z_1$, for $i=1,2$, and $\text{J}_1\text{J}_2\star z_1$. Additionally, the same condition implies that if the Jacobian of $\Phi_\theta$ at $z_1$ is $2/n$, the $\text{J}_1\text{J}_2\star z_1$ has the same Jacobian, while the Jacobian equals $-2/n$ at $\text{J}_i\star z_1$, $i=1,2$. Similarly to Definition \ref{induced basic j1 flow}, we can define \textit{induced basic $(j_1,j_2)$-flows} on $\mathcal{S}_3$. Furthermore, similarly to Lemma \ref{relations <-> basic flow} it can be shown that starting from an induced basic $(j_1,j_2)$-flow on $\mathcal{S}_3$ we can get a pair $(\Phi_\theta,\tilde{f})$ satisfying the Uchida conditions, where $\tilde{f}:\mathcal{S}_3 \rightarrow \mathbb{R}\text{P}^1$. Consequently, by the results in \cite{Uchida}, we get an analytic $G$ action on $\text{S}^{p+q-1}$ extending the orthogonal $K$ action. We will refer to an action of $G$ on $\text{S}^{p+q-1}$ arising this way as an \textit{action from the Uchida construction corresponding to a basic} $(\textup{J}_1,\textup{J}_2)$-\textit{flow}. Similar arguments to those we will use to prove Theorem \ref{basic axns thm} can be used to prove a similar result in this case, see also Remark \ref{remark classification of basic j1 fl}:
\begin{theorem} \label{basic Uchida axns thm}
\textup{(see also \cite[Theorem]{Uchida})} Two actions from the Uchida construction corresponding to basic $(\textup{J}_1,\textup{J}_2)$-flows are analytically isomorphic if and only if the corresponding basic $(\textup{J}_1,\textup{J}_2)$-flows are analytically isomorphic. Moreover, the basic $(\textup{J}_1,\textup{J}_2)$-flows are classified up to analytic isomorphism by the Jacobian at the fixed points, which is $\pm 2/n$, for $n \in \mathbb{N}$, and the global invariant $\mu$.
\end{theorem}
 Note that this theorem strengthens Uchida's result about $G$ actions on $\text{S}^{p+q-1}$ extending the orthogonal $K$ action, in the analytic setting.

\subsection{Adapted lemmas from Uchida}

Suppose $\Phi_\theta$ is an induced basic $j_1$-flow on $\mathcal{S}$ and $f = f_\Phi$, see Definition \ref{induced basic j1 flow} and Remark \ref{f_phi def remark}. Recall that by Lemma \ref{relations <-> basic flow}, $\Phi$ and $f$ satisfy relation (A1)-(A4) from Remark \ref{(A1)-(A4) props remark}.

For $z \in \mathcal{S}$, let
\begin{equation} \label{P(z) eq}
P(z) = \frac{1}{f(z)^2+1} \left( f(z) \, e_1 +  \, e_{p+1} \right)^T \cdot \left( f(z) \, e_1 +  \, e_{p+1} \right) \in \mathbb{R}^{(p+q) \times (p+q)} 
\end{equation}
Then, define the following subgroup of $G$
\[ U\left(P(z) \right) = \left\{ g \in G : gP(z)g^T = P(z) \right\} \]
Recall $U(z)$ by (\ref{U(z) eq definition}) and note that $U\left(P(z) \right)^\circ = U(z)$, see \cite{Uchida}. Straightforward matrix multiplication and (A3) show that, for $z \in \mathcal{S}$,
\begin{equation} \label{m(theta) and P}
    m(\theta)P(z)m(\theta) = \lambda(\theta,z) P \left( \Phi_\theta(z) \right)
\end{equation}
and hence
\begin{equation} \label{m(theta) and U}
m(-\theta)U(z)m(\theta)=U\left( \Phi_\theta (z) \right)    
\end{equation}
See also \cite{Uchida}.
We present three adaptations of lemmas from \cite{Uchida} that hold in our setting. The proofs are similar to \cite{Uchida}.

\begin{lemma}\textup{(see also \cite[Lemma 3.5]{Uchida})} \label{adapted lemma 1}
    Suppose $kP(z)k^T=P(w)$ for some $k \in K$ and $z,w \in \mathcal{F}$. Then 
    \begin{enumerate}
        \item[(1)] If $f(z) \neq 0$, then $f(z)=f(w)$ and $k\in H \bigcup j_1j_2H$ or $f(z) = f(j_i(w))$ and $k \in j_1H \bigcup j_2H$.
        \item[(2)] If $f(z)=0$, then $f(z)=f(w)$ and $k \in U(z) \bigcup j_1j_2U(z)$. 
    \end{enumerate}
\end{lemma}
\begin{lemma}\textup{(see also \cite[Lemma 3.6]{Uchida})} \label{adapted lemma 2}
    If $z \in \mathcal{S}$ and $f(\Phi_\theta (z) ) = f(j_i(z))$, then $|f(z)|\neq 1$ and $\Phi_\theta (z) = j_1(z)$.
\end{lemma}
\begin{lemma}{\textup{(\cite[Lemma 3.7]{Uchida})}} \label{adapted lemma 3}
    If $j_im(\theta) \in U(z)$, then $|f(z)| \neq 1$ and $i=1$.
\end{lemma}
 Note that in our case $i=2$ cannot happen. Indeed, then we would get $(1+f^2(z)) = (f^2(z) -1)\text{cosh}(\theta)$ which is impossible, since the left hand side is always positive, while the right hand side is always negative.

\subsection{Well-definedness of the $\text{SO}^\circ(p,q)$ actions} \label{subsection: well-def of axns}

Firstly, we show that the definition of the action in (\ref{axn def}) of an element $g$ of $G$, assuming $z\in \mathcal{S}$ and $k_0 = (\kappa_1,\kappa_2) \in K$ are fixed, is independent of the expression of $gk_0$ using (\ref{eq:1}). To that end, suppose that $gk_0 = k m(\theta) u = k' m(\theta ') u'$, for $k, k' \in K$, $\theta, \theta ' \in \mathbb{R}$ and $u, u' \in U(z)$. For $\theta \in \mathbb{R}$, let
\begin{equation} \label{lambda eq}
    \lambda(\theta,z) = \frac{1}{f^2(z)+1} \left[ \left( f(z) \, \text{cosh}(\theta) + \text{sinh}(\theta) \right)^2 + \left( f(z) \, \text{sinh}(\theta) + \text{cosh}(\theta) \right)^2 \right]
\end{equation}
We have
\begin{align*}
    km(\theta) P(z) m(\theta) k^T & = k' m(\theta ') P(z) m(\theta ') (k')^T \\
    \Rightarrow \lambda(\theta, z) k P\big( \Phi_\theta (z) \big) k^T & = \lambda(\theta ', z) k' P\big( \Phi_{\theta '} (z) \big) (k')^T 
\end{align*}
by (\ref{m(theta) and P}). By taking the trace of both sides, we get 
\begin{equation} \label{eq: lambda}
    \lambda(\theta, z) = \lambda(\theta ',z)
\end{equation}
and 
\begin{equation} \label{eq:3}
    k P\big( \Phi_\theta (z) \big) k^T  = k' P\big( \Phi_{\theta '} (z) \big) (k')^T
\end{equation}
Now, (\ref{eq:3}) gives $\displaystyle (k')^T k P\big( \Phi_\theta (z) \big) k^T k' = P\big( \Phi_{\theta '} (z) \big) $, which by Lemma \ref{adapted lemma 1} and Lemma \ref{adapted lemma 2} implies
\begin{equation*}
\begin{cases}
    f \big( \Phi_\theta (z) \big) = f \big( \Phi_{\theta'} (z) \big) & \text{or} \\
    f \big( \Phi_\theta (z) \big) = - f \big( \Phi_{\theta'} (z) \big)
\end{cases}
\end{equation*}
\begin{equation} \label{eq:4}
    \Rightarrow \begin{cases}
    f \big( \Phi_{\theta - \theta '} (z) \big) = f \big( z \big) & \text{or} \\
    f \big( \Phi_{\theta + \theta'} (z) \big) = - f \big( z \big) = f(j_1(z))
\end{cases}
\end{equation}

We deal first with the first case of (\ref{eq:4}). Suppose $f \left( \Phi_{\theta - \theta '} (z) \right) = f \big( z \big)$. If $f(z)=1$, then by (\ref{eq: lambda}) we get:
\begin{align*}
      \big( \text{cosh}(\theta) + \text{sinh}(\theta) \big)^2 & = \big( \text{cosh}(\theta ') + \text{sinh}(\theta ') \big)^2 \\
    \Rightarrow ( e^{\theta} )^2  = (e^{\theta '})^2  &\Rightarrow \theta = \theta '
\end{align*}

Similarly, if $f(z)=-1$, we again get $\theta = \theta '$. On the other hand, if $|f(z)| \neq 1$, then we have
\begin{align*}
    f \big( \Phi_{\theta - \theta '} (z) \big) & = f \big( z \big) \\
    \xRightarrow{\theta_0 = \theta - \theta '} \frac{f(z) + \text{tanh}(\theta_0)}{1+ f(z)\text{tanh}(\theta_0)} & = f(z) \\
    \Rightarrow f(z) + \text{tanh}(\theta_0) & = f(z) + f^2(z) \text{tanh}(\theta_0) \\
    \Rightarrow \text{tanh}(\theta_0) & = f^2(z) \text{tanh}(\theta_0) \\
    \Rightarrow \big(1-f^2(z) \big) \text{tanh}(\theta_0) & = 0 \xRightarrow{|f(z)| \neq 1} \text{tanh}(\theta_0)  = 0 \\
   \Rightarrow \theta_0  = 0 & 
    \Rightarrow \theta  = \theta '
\end{align*}

Having now that $\theta = \theta '$, we have that $k^{-1}k'=m(\theta)u(u')^{-1}m(\theta) \in m(\theta)U(z)m(\theta)^{-1}=U \big(\Phi_\theta (z) \big)$
by (\ref{m(theta) and U}). Now, suppose an element $\tilde{k}=(\kappa_1 , \kappa_2 ) \in K$ is in $U(\zeta)$ for some $\zeta \in$ S$^1$. Then, $\kappa_1  \left( f(\zeta)e_1 \right) = f(\zeta) e_1 \text{ and } \kappa_2 e_{p+1} = e_{p+1}$. Thus, $\kappa_2 \in\text{SO}(q-1)$ and, if $f(\zeta) \neq 0$, then $\kappa_1 \in\text{SO}(p-1)$, while if $f(\zeta)=0$, then $\zeta = (\pm e_{p+1}, \epsilon_1)$
In any case we see that $K \cap U(\zeta) = \text{Stab}_K(\zeta)$. Therefore, $k^{-1}k' \in$ Stab$_K \left( \Phi_\theta (z) \right)$ and hence
\begin{equation*}
    k \star \Phi_\theta (z) = k' \star \Phi_{\theta }(z)
\end{equation*}

 Now, we consider the second case of (\ref{eq:4}), namely the case $f \big( \Phi_{\theta + \theta'} (z) \big) = f(j_1(z))$. By equation (\ref{eq:3}) and Lemma \ref{adapted lemma 1}, we get that $k^{-1}k' \in j_1H \cup j_2H$. Hence, there exists $h \in H$ for which $k' = kj_ih$. Then
\begin{align*}
    km(\theta)u  = k' m(\theta ') u' 
  &\Rightarrow  m(\theta) u  = j_ihm(\theta ') u'
    \Rightarrow m(\theta) u  = j_im(\theta ')hu' \\
    \Rightarrow m(\theta) u  = m(- \theta ')j_ihu'
    &\Rightarrow j_i m(\theta + \theta ')  = hu'u^{-1} 
\end{align*}
Hence, $j_i m(\theta + \theta ') \in U(z)$ and by Lemma \ref{adapted lemma 3}, $i=1$ and $| f(z) | \neq 1$, which then implies $\Phi_{\theta + \theta '} (z) = j_1(z)$. As a result, we have 
\begin{align*}
    k' \star \Phi_{\theta '} (z) & = kj_1h \star \Phi_{\theta '} (z)
      = kj_1 \star \Phi_{\theta '}(z) \\
     & = kj_1 \star \Phi_{-\theta + (\theta + \theta ')}(z) 
      = kj_1m(- \theta) \star \Phi_{\theta + \theta '}(z) \\
     & = km(\theta) j_1m(\theta + \theta ') \star z 
      = km(\theta) \star z 
      = k \star \Phi_\theta (z) \\
\end{align*}
Hence, the action in (\ref{axn def}) is well defined when $k_0=(\kappa_1, \kappa_2) \in K$ and $z \in \mathcal{S}$ are fixed.

Next, suppose $(v,w) \in \text{S}^p \times \text{S}^{q-1}$. We now show that the action in (\ref{axn def}) is independent of the way $(v,w)$ is written as $(v,w)=k_0 \star z$ for $k_0 \in K$ and $z \in \mathcal{S}$. To the end, assume $(v,w) = k_0 \star z = k_0' \star z'$, for $k_0=(\kappa_1,\kappa_2), \, k_0'= (\kappa_1 ',\kappa_2 ') \in K$ and $z=(\alpha_ze_1+\beta_ze_{p+1}, \epsilon_1), \, z'=(\alpha_{z'}e_1+\beta_{z'}e_p+1, \epsilon_1) \in \mathcal{S}$. Moreover, for a $g \in G$, let $gk_0= km(\theta)u$ and $gk_0'= k'm(\theta ')u'$. Recall that for $\kappa \in \text{SO}(p)$, $\displaystyle \tilde{\kappa} = 
    \begin{bmatrix}
        \kappa & \\ & 1
    \end{bmatrix}$, see equation (\ref{SOp embed in SOp+1}). Now, 
\begin{align*}
     k_0 \star z  = k_0' \star z' 
    \Rightarrow  \begin{cases}
        \tilde{\kappa_1}(\alpha_ze_1+\beta_ze_{p+1}) = \tilde{\kappa_1 '}(\alpha_{z'}e_1+\beta_{z'}e_{p+1}) \\
        \kappa_2 \epsilon_1 = \kappa_2' \epsilon_1
    \end{cases}
\end{align*}
From that we deduce that $\beta_z = \beta_{z'}$ and that $\alpha_z = \pm \alpha_{z'}$, since $\alpha_z^2+\beta_z^2=\alpha_{z'}^2+\beta_{z'}^2=1$.  Hence, $z=j_1^\eta z'$, where $\eta = 0$ or 1. Hence, $(k_0')^{-1}k_0j_1^\eta$ fixes $z'$ and therefore, $(k_0')^{-1}k_0j_1^\eta \in U(z') $, which implies $ k_0' = k_0j_1^\eta u'' $, for some $u'' \in U(z')$. Then,
\begin{align*}
    gk_0' & = gk_0j_1^{\eta}u'' 
     = km(\theta)uj_1^{\eta}u'' 
     = km(\theta)j_1^\eta \big( j_1^\eta u j_1^\eta u'' \big) 
     = kj_1^\eta m\big( (-1)^\eta \theta \big) \big( j_1^\eta u j_1^\eta u'' \big)
\end{align*}
Now, $P(j_i(z)) = j_iP(z)j_i $ and therefore, $ j_1^\eta u j_1^\eta \in U(z')$. But then, \\
$kj_1^\eta m\big( (-1)^\eta \theta \big) \big( j_1^\eta u j_1^\eta u'' \big)$ and $k'm(\theta ')u'$ are different expressions for $gk_0'$ using (\ref{eq:1}) for the point $z' \in \mathcal{S}$ and as we've already seen the action defined in (\ref{axn def}) agrees for the two expressions, namely $kj_1^\eta \star \Phi_{(-1)^\eta \theta} (z') = k' \star \Phi_{\theta '}(z')$. If $\eta = 0$, then $ z=z'$ and $ k' \star \Phi_{\theta '}(z') = k \star \Phi_{ \theta} (z)$, while if $\eta=1$, then $z=j_1z'$ and $k' \star \Phi_{\theta '}(z') = kj_1 \star \Phi_{- \theta} (z') = k \star \Phi_\theta (j_1z') = k \star \Phi_\theta (z)$.

Finally, we see that this definition does indeed give an action of $G$. Let $g, g' \in G$, $(v,w) \in$ S$^p \times$ S$^{q-1}$ and $(v,w) = (\kappa_1,\kappa_2) \star z$, and write $g(\kappa_1,\kappa_2) = km(\theta)u$ and $g'k = k'm(\theta ')u'$, where $u \in U(z)$ and $u' \in U(\Phi_\theta (z))$. Then,
\begin{align*}
    g'g(\kappa_1, \kappa_2) & = g' km(\theta)u 
     = k' m(\theta ') u' m(\theta)u 
     = k' m(\theta ' + \theta) \big( m(- \theta)u'm(\theta) \big) u 
\end{align*}
But, by (\ref{m(theta) and U}), $m(- \theta)u'm(\theta) \in U(z)$, therefore
\begin{align*}
    g' \star \Big( g \star z \Big) & = g' \star \Big( k \star \Phi_\theta (z) \Big) 
     = k' \star \Phi_{\theta '} (\Phi_\theta (z)) 
     = k' \star \Phi_{\theta ' + \theta} (z) 
     = g'g \star z
\end{align*}
It is also easy to see that this action extends the action of $K$ on S$^p \times$ S$^{q-1}$.

\subsection{Isomorphic Actions} \label{isom axns}
Here we prove Theorem \ref{basic axns thm}
\begin{proof}{(of Theorem \ref{basic axns thm})}
Suppose that for two $G$-actions from the basic construction, the basic $\textup{J}_1$-flows on $\text{S}^1$ that they correspond to, say $\widetilde{\Phi}^1_\theta$ and $\widetilde{\Phi}^2_\theta$, are analytically isomorphic. Let $\Phi^1_\theta$ be the induced basic $j_1$-flow of $\widetilde{\Phi}^1_\theta$ on $\mathcal{S}$, see Definition \ref{induced basic j1 flow}. Similarly, consider the induced basic $j_1$-flow $\Phi^2_\theta$ of $\widetilde{\Phi}^2_\theta$.  Recall $\mathcal{S} = \{ (\alpha e_1 + \beta e_{p+1}, \epsilon_1 )\} \subseteq \text{S}^p \times \text{S}^{q-1}$. Now, since $\widetilde{\Phi}^1_\theta$ and $\widetilde{\Phi}^2_\theta$ are analytically isomorphic, it is immediate that $\Phi^1_\theta$ and $\Phi^2_\theta$ are also analytically isomorphic. Equivalently, the vector fields that $\Phi^1_\theta$ and $\Phi^2_\theta$ generate, say $X$ and $Y$ respectively, are isomorphic.
We note that, by the result in \cite{HITCHIN1991359}, analytic vector fields on the circle are classified, see also Remark \ref{remark classification of basic j1 fl}. Therefore, there exists an analytic isomorphism, $\Psi$, of $\mathcal{S}$ such that $\Psi \circ \Phi^1_\theta = \Phi^2_\theta \circ \Psi$. Moreover, we can assume that $\Psi$ satisfies $\Psi \circ j_1 = j_1 \circ \Psi$, since $\Phi_\theta^{i}$, $i=1,2$, satisfy that relation with $j_1$, see \cite{Lentas_thesis} for more details. Let $f_1 = f_{\Phi^1}$ and $f_2 = f_{\Phi^2}$, see Remark \ref{f_phi def remark}. We claim that $\Psi$ also satisfies $f_1 = f_2 \circ \Psi$. Indeed, if $X= g_X \cdot \partial_1$, where $\partial_1$ is a basic vector of $\mathcal{S}$, then by (A1) and (A3) of Remark \ref{(A1)-(A4) props remark}, it is easy to see that

\begin{equation*}
    \begin{cases}
     g_X \circ j_1 = g_X \\
     g_X \cdot  (f_1)' = (f_1)^2 -1
    \end{cases}
\end{equation*}
and similarly for $Y$.

Let $z = \left( \pm e_{p+1}, \epsilon_1 \right)$. Then, $f_1(z)=0$ and $g_X(z) \neq 0$. By $\Psi \circ j_1 = j_1 \circ \Psi$, $z' = \Psi (z)$ is also fixed by $j_1$, i.e. it is either $(e_{p+1}, \epsilon_1)$ or $(-e_{p+1}, \epsilon_1)$ and hence, $f_2 \left( z' \right) = 0$ and $g_Y \left( z' \right) \neq 0$.
Let $I$ be a neighbourhood of $z$ and see $I$ as an interval on the real line. Similarly for an $I'$ around $z'$, such that $\Psi: I \rightarrow I'$ analytic isomorphism.

In terms of $g_X$ and $g_Y$, the relation $\Psi \circ \Phi^1 = \Phi^2 \circ \Psi$ becomes $g_X \cdot \Psi' = g_Y \circ \Psi$, which is equivalent to $g_X \circ \Psi^{-1} = g_Y \cdot \big( \Psi^{-1} \big)'$. Now, 
\begin{align*}
    & g_X \cdot  f_1' = f_1^2 -1 
    \Rightarrow  g_X \circ \Psi^{-1} \cdot f_1' \circ \Psi^{-1} = f^2_1 \big( \Psi^{-1} \big) - 1 \\
    \Rightarrow & g_Y \cdot \big( \Psi^{-1} \big)' \cdot f_1' \circ \Psi^{-1} = f^2_1 \big( \Psi^{-1} \big) - 1
    \Rightarrow g_Y \cdot \big( f_1 \circ \Psi^{-1} \big)' = f^2_1 \big( \Psi^{-1} \big) - 1 
\end{align*}
Therefore, both $f_2$ and $f_1 \circ \Psi^{-1}$ solve the equation
\begin{equation*}
    \begin{cases}
        u' = \frac{1}{g_Y} \left( u^2 -1 \right)\\
        u(z') = 0
    \end{cases}
\end{equation*}
where by $u'$ we denote the derivative of $u$. Therefore they're equal around $z'$, and being analytic, they're equal. Therefore, we have an analytic isomorphism $\Psi : \mathcal{S} \rightarrow \mathcal{S}$ such that $\Psi \circ \Phi^1_\theta = \Phi^2_\theta \circ \Psi \text{ and } f_1 = f_2 \circ \Psi$

 Now, we can define a map $\tilde{\Psi}: \text{S}^p \times \text{S}^{q-1} \rightarrow \text{S}^p \times \text{S}^{q-1}$ by $ k \star \Phi^1_\theta(z) \mapsto k \star \Phi^2_\theta \left( \Psi(z) \right)$
\begin{itemize}
    \item[$\bullet$] $\tilde{\Psi}$ is well defined: \\
    Suppose $k \star \Phi^1_\theta (z) = k' \star \Phi^1_{\theta '} (z')$. If $\Phi^1_\theta (z) = ( \pm e_{p+1}, \epsilon_1)$, then by $ \Phi^1_\theta (z) = k^{-1}k' \star \Phi^1_{\theta '}(z')$, we get that $k^{-1}k' \in \text{SO}(p) \times \text{SO}(q-1)$ and $\Phi^1_{\theta '}(z')=\Phi^1_\theta (z) = (\pm e_{p+1}, \epsilon_1)$. Then, by $0 = f_1(\pm e_{p+1}, \epsilon_1) = f_2 \left( \Psi (\pm e_{p+1}, \epsilon_1) \right) $ we have $ \Psi (\pm e_{p+1}, \epsilon_1) = (e_{p+1}, \epsilon_1)$ or $(-e_{p+1}, \epsilon_1)$, hence $k \star \Psi (\pm e_{p+1}, \epsilon_1) = k' \star \Psi (\pm e_{p+1}, \epsilon_1)$. Finally, $k \star \Psi (\pm e_{p+1}, \epsilon_1) = k' \star \Psi (\pm e_{p+1}, \epsilon_1) $ implies $ k \star \Psi (\Phi^1_\theta(z)) = k' \star \Psi (\Phi^1_{\theta ' }(z'))$. Therefore, $k \star \Phi^2_\theta \left( \Psi(z) \right) = k' \star \Phi^2_{\theta '} \left( \Psi(z') \right)$. If $\Phi^1_\theta(z) \neq (\pm e_{p+1}, \epsilon_1)$, then $\Phi^1_\theta(z') = \Phi^1_\theta(z)$ or $j_1 \star \Phi^1_\theta (z)$, and $k^{-1}k' \in H$ or $j_1H$ respectively. In the former case the result is immediate. In the latter, it follows by the equations $\Phi^i_\theta \circ j_1 = j_1 \circ \Phi^i_\theta$ and $\Psi \circ j_1 = j_1 \circ \Psi$.
    \item[$\bullet$] $\tilde{\Psi}$ is onto: \\
    Let $(v,w) \in \text{S}^p \times \text{S}^{q-1}$. There exist $k \in K$, $z \in \mathcal{S}$ and $\theta \in \mathbb{R}$ such that $(v,w) = k \star \Phi^2_\theta (z)$. Then, $(v,w) = \tilde{\Psi} \left( k \star \Phi^1 \left( \Psi^{-1} (z) \right) \right)$
    \item[$\bullet$] $\tilde{\Psi}$ is 1-1:
    The proof proceeds similarly to the well definedness of $\tilde{\Psi}$. 
    \item[$\bullet$] $\tilde{\Psi}$ is $G$-equivariant: The proof is identical to that of Lemma \ref{G-equiv map btw two axns} in Section \ref{axn with SOp fixed pnt}.
    \item[$\bullet$] $\tilde{\Psi}$ is a local analytic isomorphism: The proof is identical to that of Lemma \ref{local anal isom brw two axns} in Section \ref{axn with SOp fixed pnt}.
\end{itemize}
Therefore, the two $G$ actions are analytically isomorphic, if the analytic isomorphism $\Psi$ on $\mathcal{S}$ between $\Phi^1_\theta$ and $\Phi^2_\theta$ exists. On the other hand, it is immediate that an analytic, $G$-equivariant isomorphism between two such $G$-actions will result in isomorphic induced basic $j_1$-flows on $\mathcal{S}$ and hence, the corresponding basic $\text{J}_1$-flows on $\text{S}^1$ are also isomorphic.
\end{proof}
\section{Extendable $\text{SO}(p) \times \text{SO}(q)$ actions} \label{SOp SOq axns}

Consider the maximal compact subgroup of $G$, $K= \text{SO}(p) \times \text{SO}(q)$. We are interested in the analytic actions of $K$ on a closed, connected manifold $M$ of dimension $p+q-1$, that extend to an analytic action of $G$. To that end, we are going to look at the possible $\text{SO}(p)$ and $\text{SO}(q)$ orbits. Uchida has classified the subgroups of $\text{O}_p$ of codimension at most $2p-2$, see \cite{Uchida1981REALAS} and \cite{Uchida_clssf_rlanalytic_SLn}. We show that Uchida's result can be applied here.
We assume $p \geq q$ and that $K$ acts on a manifold $M$ as above. Assume also that the action extends to $G$. Let $x \in M$.

\noindent \textbf{Notation:} We denote by $\mathcal{O}^p$ the $\text{SO}(p)$-orbit of $x$ and by $\mathcal{O}^q$ its $\text{SO}(q)$-orbit. 
\begin{lemma}
    $\textup{dim}\mathcal{O}^p < 2p-1$
\end{lemma}
\begin{proof}
Since $p \geq q$, dim$\mathcal{O}^p \leq p +q-1 \leq 2p-1$. If dim$\mathcal{O}^p$ = $2p-1$, then it follows $q=p$. Consider \text{Stab}$_{\text{SO}(q)}(x)$, which acts trivially on $\text{T}_x\mathcal{O}^p = \text{T}_xM$. If \text{Stab}$_{\text{SO}(q)}(x)$ has dimension $\geq 1$, then the action of $K$ is not locally effective and so it cannot extend to an action of $G$. If \text{Stab}$_{\text{SO}(q)}(x)$ has dimension $=0$, then $\mathcal{O}^q$ has dimension $q(q-1)/2 = p(p-1)/2$. For $p=q \geq 5$ this is bigger than the dimension of the manifold and for $p=q=4$ or 3, it is impossible for $\mathcal{O}^p$ to be of dimension $2p-1$ since that is bigger than the dimension of $\text{SO}(p)$.
\end{proof}
 So, we suppose we have dim$\mathcal{O}^p \leq 2p-2$ and therefore, we can apply Uchida's result. Let $V := $T$_x\mathcal{O}^p \bigcap $T$_x\mathcal{O}^q$. In particular, $V \leq$ T$_x\mathcal{O}^p$ is a trivial subrepresentation of the isotropy representation of $H' =\text{Stab}_{\text{SO}(p)}(x)$. Therefore, the dimension of $V$ is bounded by the dimension of a maximal subspace of $\text{T}_x \mathcal{O}^p$ on which $H'$ acts trivially via the isotropy representation. In turn, the dimension of $\text{T}_x \mathcal{O}^q$ is bounded by 
\begin{equation} \label{q ineq 1}
    \text{dim} \text{T}_x \mathcal{O}^q \leq \text{dim}M - \text{dim} \text{T}_x \mathcal{O}^p + \text{dim}V 
 \end{equation}
At the same time $q$ must satisfy
\begin{equation} \label{q ineq 2}
    \text{dim} \mathcal{O}^p \leq p+q-1 \quad \text{and} \quad q \leq p
    \end{equation} 

 We present Uchida's classification of the orbit types, $\text{SO}(p) \big/ H'$, mentioned above, in Table \ref{table:1} below. The last column comes from taking (\ref{q ineq 2}) into consideration.

\begin{table}[h] \footnotesize
    \centering
\caption{Uchida Table}
\label{table:1}
    \begin{tabular}{|c|c|c|c|c|} 
        \hline
$p$ & subgroup $H'$ & dim$\mathcal{O}^p$ & dim$M$ & $q$   \\
\hline  \hline
$p$ & $\text{SO}(p-2)$ & $2p-3$ & $\leq 2p-1$ & $p-2 \leq q \leq p$ \\
$p$ & $\text{SO}(p-2) \times \text{SO}(2)$ & $2p-4$ & $\leq 2p-1$ & $p-3 \leq q \leq p$ \\
9 & Spin$(7)$ & 15 & $\leq 17$ & $7 \leq q \leq 9$   \\
8 & $\text{G}_2$ & 14 & $\leq 15$ & $7 \leq q \leq 8$\\
8 & $\text{U}_4$ & 12 & $\leq 15$ & $5 \leq q \leq 8$ \\
8 & $\text{SU}_4$ & 13 & $\leq 15$ & $6 \leq q \leq 8$ \\
7 & $\text{G}_2$ & 7 & $\leq 13$ & $3 \leq q \leq 7$\\
7 & $\text{U}_3$ & 12 & $\leq 13$ & $6 \leq q \leq 7$\\
7 & $\text{SO}(3) \times \text{SO}(4)$ & 12 & $\leq 13$ & $6 \leq q \leq 7$ \\ 
6 & $\text{SO}(3) \times \text{SO}(3)$ & 9 & $\leq 11$ & $4 \leq q \leq 6$ \\
6 & $\text{U}_3$ & 6 & $\leq 11$ & $3 \leq q \leq 6$ \\
6 & $\text{SU}_3$ & 7 & $\leq 11$ & $3 \leq q \leq 6$\\ 
6 & $\text{U}_2 \times \text{U}_1$ & 10 & $\leq 11$ & $5 \leq q \leq 6$ \\
5 & $\text{U}_2$ & 6 & $\leq 9$ & $3 \leq q \leq 5$ \\
5 & $\text{SU}_2$ & 7 & $\leq 9$ & $3 \leq q \leq 5$ \\
5 & $\text{U}_1 \times \text{U}_1$ & 8 & $\leq 9$ & $4 \leq q \leq 5$\\
5 & $\text{SO}(3)$ & 7 & $\leq 9$ & $3 \leq q \leq 5$  \\
3 & $\{ 1 \}$ & 3 & $\leq 5$ & 3 \\
\hline 
$p$ & $\text{SO}(p-1)$ & $p-1$ & $\leq 2p-1$ & $3 \leq q \leq p$ \\
8 & Spin$(7)$ & 7 & $\leq 15$ & $3 \leq q \leq 8$ \\
4 & $\text{SU}_2$ & 3 & $\leq 7$ & $3 \leq q \leq 4$ \\
4 & $\text{U}_2$ & 2 & $\leq 7$ & $3 \leq q \leq 4$  \\
\hline
    \end{tabular}

\end{table}

 For reference, we will call the last four rows of Table \ref{table:1} the bottom part and the rest the top part. The embedding of $H'$ in $\text{SO}(p)$ in each case is the usual one, except for the penultimate row of the top part of Table \ref{table:1} where it is the irreducible 5-dimensional representation of $\text{SO}(3)$.
\newpage
\begin{lemma}
    In all the cases for $\mathcal{O}^p$, the dimension of a maximal subspace of $\text{T}_x \mathcal{O}^p$ on which the respective subgroup $H'$ acts trivially is at most 1. In particular, $V$ is at most 1 dimensional.
\end{lemma}
\begin{proof}
    This is easy to check in all cases, except maybe for the penultimate row of the top part of Table \ref{table:1}.
Let $v \in \text{T}_x\mathcal{O}^p$, nonzero, such that $H'$ acts trivially on $v$ in the isotropy representation. Since we can identify $\text{T}_x \mathcal{O}^p$ with $\mathfrak{so}_5 \Big/ \mathfrak{h}'$, where $\mathfrak{h}'$ is the Lie algebra of $H'$, we can choose an $X \in \mathfrak{so}_5$ such that $X \, \text{mod} \mathfrak{h}' \equiv v$ under this identification, and $e^{tX} \in N_{\text{SO}(5)}( H')$, the normaliser of $H'$ in $\text{SO}(5)$. Now, if there exists a subspace $\Tilde{V} \leq \text{T}_x \mathcal{O}^p$ on which $H'$ acts trivially and dim$\Tilde{V} \geq 2 $, then for $N_{\text{SO}(5)}( H')$ we would have
\[ \text{dim}N_{\text{SO}(5)}( H') \geq \text{dim}H' + 2 \Rightarrow \text{dim}N_{\text{SO}(5)}( H') \geq 5 \]
But then, $N_{\text{SO}(5)}( H')$ is a subgroup of $\text{SO}(5)$ of codimension at most 5. According to Table \ref{table:1}, the only possibility is $N_{\text{SO}(5)}( H') \simeq \text{SO}(4)$ and (by conjugating) we can assume  $N_{\text{SO}(5)}( H')$ is imbedded in $\text{SO}(5)$ with the standard representation of $\text{SO}(4)$. But then, since $\text{SO}(3) \leq N_{\text{SO}(5)}( H')$, the representation of $\text{SO}(3)$ in $\text{SO}(5)$ is not irreducible, which is a contradiction. So, the dimension of a subspace of $\text{T}_x \mathcal{O}^p$ on which $H'$ acts trivially is at most 1.
\end{proof}

 Immediate from Table \ref{table:1} is the following:
\begin{lemma} \label{ Oq 1-dim}
    $\mathcal{O}^q$ cannot be 1-dimensional.
\end{lemma}
 Moreover, we have
\begin{lemma} \label{SOq fix pnt}
    Suppose that 
    \textup{dim}$M$ - \textup{dim}$\mathcal{O}^p \leq q-1$. Then, $x$ cannot be fixed by $\textup{SO}(q)$, i.e. $\mathcal{O}^q$ cannot be 0-dimensional.
\end{lemma}
\begin{proof}
    Firstly, let $q \neq 4$. Assume SO$_q$ stabilises $x$ and consider its isotropy representation on T$_xM$. Then, $\text{SO}(q)$ acts trivially on T$_x\mathcal{O}^p$ since $\text{SO}(p)$ and $\text{SO}(q)$ commute in $G$. Moreover, since it does not have a non-trivial representation of dimension $\leq q-1$, we conclude that its isotropy representation is trivial. But then, SO$_q$ acts trivially on a neighbourhood of $x$ in $M$, which contradicts our assumption that the action extends to $\text{SO}^\circ(p,q)$ since $\text{SO}^\circ(p,q)$ is simple and hence the action is locally effective.
    
     For $q=4$, there is the possibility that dim$M$ - dim$\mathcal{O}^p =3$ and that the representation of $\text{SO}(4)$ on a 3-dimensional complement of $\text{T}_x \mathcal{O}^p$ is not trivial. In that case, pick an $\text{SO}(4)$ invariant metric on $M$ and since $\text{SO}(4)$ acts trivially on $\text{T}_x \mathcal{O}^p$. The isotropy representation gives a homomorphism $\text{SO}(4) \rightarrow \text{O}_3$. Then, this homomorphism has kernel with nontrivial dimension, which contradicts the local effectiveness of the action again.
\end{proof}


\begin{lemma} \label{axn contra 1}
   Suppose 
that $\mathcal{O}^p \simeq \textup{SO}(p) / H'$ and that, in the isotropy representation of $H'$ on $\textup{T}_x \mathcal{O}^p$, the maximum dimension of a subspace on which $H'$ acts trivially is $\delta$, with $\delta = 0$ or 1. Then, the following situation is impossible:
    \begin{itemize}
        \item $(2p-1) - \textup{dim} \mathcal{O}^p \leq 3 - \delta$ 
        \item[] and 
        \item $\textup{dim} \mathcal{O}^p +1 -p \geq 5$
    \end{itemize}
\end{lemma}
\begin{proof}
    The first condition forces dim$\mathcal{O}^q \leq 3$. 
    However, the second condition forces $q \geq 5$ which contradicts $1\leq \text{dim}\mathcal{O}^q \leq 3$, and $\text{SO}(q)$ cannot fix $x$ by Lemma \ref{SOq fix pnt}.
\end{proof}


 Note that if dim$\mathcal{O}^p +$ dim$\mathcal{O}^q \geq$ dim$M$, then the $K$ action is transitive. Indeed, the $K$-orbit of $x$ is then open, closed and connected, therefore it equals $M$. It turns out that transitive actions cannot occur:
\begin{lemma} \label{transitive axns}
    Suppose that the action of $K$ on $M$ is transitive. Then, the action does not extend to $G$. 
\end{lemma}
\noindent The proof of Lemma \ref{transitive axns} is the content of Section \ref{contra for transitive axns}.

Using Lemma \ref{ Oq 1-dim}, Lemma \ref{SOq fix pnt}, Lemma \ref{axn contra 1} and Lemma \ref{transitive axns} it is easy to check for most cases that if $\mathcal{O}^p$ is any from the top part of Table \ref{table:1}, then there is no possible $\mathcal{O}^q$ for which such a $K$ action would extend to $G$. The only pairs $\left( \mathcal{O}^p, \mathcal{O}^q \right)$ which the above lemmas fail to eliminate are
$\displaystyle \left( \text{SO}(7) \Big/ \text{G}_2 \,  ,  \, \text{SO}(4) \Big/ \text{U}_2 \right)$ , $\left( \text{SO}(6) \Big/ \text{U}_3 \, , \, \text{SO}(4) \Big/ \text{U}_2 \right)$ and \\ $\left( \text{SO}(4) \Big/ (\text{SO}(2) \times \text{SO}(2)) \, , \, \text{SO}(4) \Big/ \text{U}_2 \right)$. These also do not extend to a $G$ action as it will be shown in Section \ref{non-ext non-trnstv axns}. As for when $\mathcal{O}^p$ is one from the bottom part of Table \ref{table:1}, the lemmas can be used to exclude the cases when $\mathcal{O}^q$ is from the top part of Table \ref{table:1}, but the cases where both $\mathcal{O}^p$ and $\mathcal{O}^q$ are from the bottom part still remain. These cases are treated in Section \ref{non-ext non-trnstv axns} as well.

\subsection{Transitive $\text{SO}(p) \times \text{SO}(q)$ actions} \label{contra for transitive axns}

Here we see why a transitive $\text{SO}(p) \times \text{SO}(q)$ action on a manifold $M$ of dimension $p+q-1$ does not extend to an action of $G= \text{SO}^\circ(p,q)$. Assume it does. Then, we get a compact homogeneous space of $G$. Let $H'$ be a subgroup of $G$ such that $\displaystyle G \big/ H' \cong M$. That means that $H'$ is a cocompact subgroup of $G$. We will need the following version of a theorem of Witte which we state for convenience:
\begin{theorem*}{(\cite[Main Theorem 1.2]{witte_cocompactsubgrps})}
    Let $G$ be a connected, semisimple Lie group with finite center and $H' \leq G$ closed and cocompact. Then, there exists a parabolic subgroup $P$ of $G$ such that, if $P^\circ = LEAN$ is the refined Langlands decomposition (see below) of $P^\circ$, then there exist closed, connected subgroup $Y \leq EA$ and a normal, closed, connected subgroup $X \leq L$ such that $(H')^\circ=XYN$
\end{theorem*}

 The \textit{refined Langlands decomposition} of $P^\circ$, as defined in \cite{witte_cocompactsubgrps}, follows from the Langlands decomposition, $P^\circ = MAN$, of $P^\circ$, by considering the product, $L$, of all the noncompact simple factors of $P^\circ$ and the maximal compact factor, $E$, of $P^\circ$.

We follow \cite{Knapp} for the definitions and notation about the roots and the root spaces of $G$, and \cite{Warner} about the definitions and notation about parabolic subgroups. We show below that, because of the dimension of the manifolds we consider, $H'$ could only be in a maximal parabolic of $G$. There are two different kind of maximal parabolic subgroups of $G$. In the standard representation of $G$ on $\mathbb{R}^{p+q}$ with a scalar product of type $(p,q)$, the first kind preserves a null line and the second kind preserves a maximal isotropic subspace. 

\subsubsection{The first kind of maximal parabolics} \label{firstkindparab}
If $\mathfrak{so}(p,q) = \left\{ \begin{bmatrix}
    X & A \\ A^T & Y
\end{bmatrix}  \right\}$ with $X,Y$ antisymmetric matrices, then a maximal abelian subspace, $\mathfrak{a}$, of the symmetric part is the collection of matrices with $X=0, Y=0$ and $A$ being of the form $ \displaystyle A = \begin{bmatrix}
    & \vdots & \\
    && a_q \\ 
    & \reflectbox{$\ddots$} & \\
    a_1 && \\
\end{bmatrix}$, where everything except for the $a_i$'s is 0. Let $f_i$ denote the restricted root that picks out the $a_i$ component of an element of $\mathfrak{a}$.

First the case $p=q$. A fundamental root system is
\begin{equation*}
    \mathcal{Y} = \{ f_1 - f_2 , \cdots , f_{p-1}-f_{p} , f_{p-1}+f_p \}
\end{equation*}
and the restricted root spaces $\mathfrak{g}^{\pm f_i \pm f_j} $, $i \neq j$, have dimension 1. Subsets of $\mathcal{Y}$ correspond to parabolic subgroups $P_\Theta$ of $G$ and vice versa, see \cite{Warner}. 
Firstly, we take 
\begin{equation*}
    \Theta = \{ f_2 - f_3 , \cdots , f_{p-1}-f_{p} , f_{p-1}+f_p \}
\end{equation*}
Denote by $\Sigma^+$ the space of positive roots, where positivity is induced by $\mathcal{Y}$ in the standard way. Moreover, denote by $\langle \Theta \rangle^+$, respectively $\langle \Theta \rangle^-$, the space of positive, respectively negative, roots that also belong in the span of $\Theta$. Next, define $\mathfrak{a}_\Theta$ to be the subspace of $\mathfrak{a}$ where all roots are 0, $\mathfrak{m}$ to be the centraliser of $\mathfrak{a}$ in $\mathfrak{k}= \mathfrak{so}(p) \oplus \mathfrak{so}(q)$, $\displaystyle \mathfrak{n}^\pm(\Theta) = \sum_{\lambda}\mathfrak{g}^{\lambda}$, for $\lambda \in \langle \Theta \rangle^\pm$, and $\displaystyle \mathfrak{n}^+_{\Theta} = \sum_{\lambda}\mathfrak{g}^{\lambda}$, for $\lambda \in \Sigma^+ - \langle \Theta \rangle^+$. In our case, $\mathfrak{a}_\Theta$ is 1-dimensional since it is the span of the element with $a_2=a_3=\cdots=a_p =0$, while $\mathfrak{m}=0$ since $p=q$, see \cite[Chapter VI, Section 4]{Knapp}. Moreover,
\begin{equation*}
    \mathfrak{n}^+(\Theta) = \sum_{2\leq i < j \leq p}\mathfrak{g}^{f_i \pm f_j} \quad \text{and} \quad
    \mathfrak{n}^+_{\Theta} = \sum_{2\leq i  \leq p}\mathfrak{g}^{f_1 \pm f_j}
\end{equation*}
Then, if $\mathfrak{p}_\Theta$ is the Lie algebra of the parabolic subgroup $P_\Theta$, we have that
\[ |\mathfrak{p}_\Theta| = |\mathfrak{m}| +|\mathfrak{a}|+|\mathfrak{n}^+_{\Theta}|+|\mathfrak{n}^+({\Theta})| + |\mathfrak{n}^-(\Theta)|\]
where, by $|\cdot|$ we mean the dimension of the respective vector space, see \cite[Chapter 1, Section 1.2.4]{Warner}. We note that, 
\[ |\mathfrak{n}^+({\Theta})| = |\mathfrak{n}^-(\Theta)| \]
Therefore, we have the following data
\begin{center}
\begin{tabular}{ |c|c| } 
\hline
$| \mathfrak{n}^+(\Theta)|$ & $=p^2-3p+2$ \\
$|\mathfrak{n}^+_{\Theta}|$ & $=2(p-1)$ \\
$|\mathfrak{p}_\Theta| = |\mathfrak{m}| +|\mathfrak{a}|+|\mathfrak{n}^+_{\Theta}|+|\mathfrak{n}^+({\Theta})| + |\mathfrak{n}^-(\Theta)|$ & $=2p^2-3p+2$ \\
\hline
\end{tabular}
\end{center}

 As for $p>q$, we have $\mathfrak{m} = \mathfrak{so}_{p-q}$. As fundamental root system we take
\begin{equation*}
    \mathcal{Y} = \{ f_1-f_2, \cdots, f_{q-1} - f_q, f_q \}
\end{equation*}
The restricted root spaces $\mathfrak{g}^{\pm f_i \pm f_j}$ have dimension 1, while the $\mathfrak{g}^{\pm f_i}$'s have dimension $p-q$. As $\Theta$ we take
\begin{equation*}
    \Theta = \{ f_2-f_3, \cdots , f_{q-1}-f_q, f_q \}
\end{equation*}
Then, $\mathfrak{a}_{\Theta}$ is 1-dimensional, spanned by the vector with $a_1=1$. Moreover, we have
\begin{equation*}
    \mathfrak{n}^+(\Theta) = \sum_{2\leq i < j \leq q} \mathfrak{g}^{f_i \pm f_j} \oplus \sum_{2\leq i \leq q}\mathfrak{g}^{f_i} \quad \text{and} \quad
    \mathfrak{n}^+_{\Theta} = \sum_{2\leq i \leq q}\mathfrak{g}  ^{f_1 \pm f_i} \oplus \mathfrak{g}^{f_1}
\end{equation*}
Hence,
\begin{center}
\begin{tabular}{ |c|c| } 
\hline
$|\mathfrak{m}|$ & $=\frac{(p-q)(p-q-1)}{2}$ \\
$| \mathfrak{n}^+(\Theta)| $ & $=pq-2q-p+2$ \\
$|\mathfrak{n}^+_{\Theta}|$ & $=p+q-2$ \\
$|\mathfrak{p}_\Theta| = |\mathfrak{m}| +|\mathfrak{a}|+|\mathfrak{n}^+_{\Theta}|+|\mathfrak{n}^+({\Theta})| + |\mathfrak{n}^-(\Theta)|$ & $=\frac{1}{2} (p^2+q^2) + pq -\frac{3}{2}p - \frac{3}{2}q+2$ \\
\hline
\end{tabular}
\end{center}

 In both cases, the codimension then of $P_\Theta$ is $p+q-2$ and so, the codimension of $H'$ in $P_\Theta$ is 1. $P_\Theta$ leaves invariant a null-line in the standard representation of $G$ in $\mathbb{R}^{p+q}$ with a scalar product of signature $(p,q)$ and hence, we know that $M_\Theta^\circ \simeq \text{SO}^\circ(p-1,q-1)$, $A_\Theta \simeq \mathbb{R}$. Suppose $A_\Theta \leq (H')^\circ$. Then, by Witte's theorem, $M_\Theta^\circ$ would have a codimension 1 normal subgroup. However, $\text{SO}^\circ(p-1,q-1)$ is semisimple, so that is impossible. Therefore, 
\[ (H')^\circ \simeq M_\Theta^\circ N\]

Assume now that we have a transitive action of $\text{SO}(p) \times \text{SO}(q)$ such that $H' \leq P_\Theta$. If $\mathcal{O}^p \simeq \text{SO}(p) \big/ H_1$ and $\mathcal{O}^q \simeq \text{SO}(q) \big/ H_2 $, since $Y := H_1 \times H_2 \leq H'$ is compact and connected, it necessarily lies in a maximal compact subgroup of $(H')^\circ$ which is isomorphic to $\text{SO}(p-1) \times \text{SO}(q-1)$. Therefore, $Y$ can be conjugated into a subgroup of $\text{SO}(p-1) \times \text{SO}(q-1)$ and we can assume $Y \leq \text{SO}(p-1) \times \text{SO}(q-1)$.
In order to get a contradiction we are going to need \textit{Goursat's lemma}:
\begin{lemma*}{(Goursat's lemma)}
    Let $G_1, G_2$ be groups and suppose $A \leq G_1 \times G_2$. Denote $p_1:A \rightarrow G_1$ and $p_2: A \rightarrow G_2$ the projections to first and second factor and let $N_1$ be the kernel of $p_2$ and $N_2$ the kernel of $p_1$. Then, $N_1$ and $N_2$ can be seen as normal subgroups of $p_1(A)$ and $p_2(A)$ respectively, and the image of $A$ in $\displaystyle p_1(A) \big/ N_1 \times p_2(A) \big/ N_2$ is the graph of an isomorphism between $p_1(A) \big/ N_1 $ and $ p_2(A) \big/ N_2$
\end{lemma*}
 Now, as far as the dimension of $Y$ is concerned, on the one hand we have
\begin{equation} \label{dimeq}
     \text{dim}(Y) = \text{dim}(\text{SO}(p)) + \text{dim}(\text{SO}(q)) - (p+q-1)
\end{equation}
On the other hand, we also have
\begin{equation} \label{dimeqN}
     \text{dim}(Y) = \text{dim}(p_1(Y)) + \text{dim}(N_2) = \text{dim}(p_2(Y)) + \text{dim}(N_1) 
\end{equation}
Since $p_1(Y) \leq \text{SO}(p-1)$, by equality (\ref{dimeq}) above, we have 
\begin{equation} \label{dimineq}
     \text{dim}(Y) \leq \text{dim}(\text{SO}(p-1)) + \text{dim}(N_2)
\end{equation}
By (\ref{dimeq}) and (\ref{dimineq}), we get that $N_2$ is a subgroup of $\text{SO}(q)$ of codimension at most $q$ that is also contained in $\text{SO}(q-1)$. By Table \ref{table:1}, the only possiblility is $N_2 = \text{SO}(q-1)$. Therefore,
\[ N_2 = p_2(Y) = \text{SO}(q-1)\]
Then,(\ref{dimeq}) and (\ref{dimeqN}) imply that $N_1 $ is a codimension $p$ subgroup of $\text{SO}(p)$ that is also contained in $\text{SO}(p-1)$. By Table \ref{table:1}, this is impossible.


\subsubsection{The second kind of maximal parabolics} \label{2nd kind parabolics}
If $p=q$, we take
\begin{equation*}
    \Theta = \{ f_1-f_2, \cdots , f_{p-1} - f_p \}
\end{equation*}
in which case $\mathfrak{a}_\Theta$ is  1-dimensional, spanned by the element $a_1=a_2=\cdots = a_p =1$. We also have
\begin{equation*}
    \mathfrak{n}^+(\Theta) = \sum_{1\leq i < j \leq p} \mathfrak{g}^{f_i - f_{j}} \quad \text{and} \quad
    \mathfrak{n}^+_{\Theta} = \sum_{1\leq i < j\leq p}\mathfrak{g}^{f_i + f_{j}}
\end{equation*}
As a result, we get
\begin{center}
\begin{tabular}{ |c|c| } 
\hline
$| \mathfrak{n}^+(\Theta)|$ & $=\frac{1}{2}(p^2-p)$ \\
$|\mathfrak{n}^+_{\Theta}|$ & $=\frac{1}{2}(p^2-p)$ \\
$|\mathfrak{p}_\Theta| = |\mathfrak{m}| +|\mathfrak{a}|+|\mathfrak{n}^+_{\Theta}|+|\mathfrak{n}^+({\Theta})| + |\mathfrak{n}^-(\Theta)|$ & $=\frac{1}{2}(3p^2-p)$ \\
\hline
\end{tabular}
\end{center}
Then, the codimension of $P_{\Theta}$ in $G$ is $\frac{1}{2}(p^2-p)$. Therefore, if $H\leq G$ is such that codim($H$) $=2p-1$, we would need 
\begin{align*}
    \text{codim}(H) & \geq \text{codim}(P_{\Theta}) 
    2p-1  \geq \frac{1}{2}(p^2-p) 
    5p-2  \geq p^2
\end{align*}
This can only happen for $p=3$ or $4$.

If $p>q$ we take as $\Theta$
\begin{equation*}
    \Theta = \{ f_1 - f_2, \cdots , f_{q-1} - f_q \}
\end{equation*}
Again $\mathfrak{a}$ is 1-dimensional, generated by the element with $a_1= \cdots a_q =1$. In addition, we have
\begin{equation*}
    \mathfrak{n}^+(\Theta) = \sum_{1 \leq i < j \leq q} \mathfrak{g}^{f_i - f_j} \quad \text{and} \quad
    \mathfrak{n}^+_{\Theta} = \sum_{1 \leq i < j \leq q}\mathfrak{g}^{f_i+f_j} \oplus \sum_{1 \leq i \leq q}\mathfrak{g}^{f_i}
\end{equation*}
Therefore, in this case we have
\begin{center}
\begin{tabular}{ |c|c| } 
\hline
$|\mathfrak{m}|$ & $=\frac{(p-q)(p-q-1)}{2}$ \\
$| \mathfrak{n}^+(\Theta)| $ & $=\frac{1}{2}(q^2-q)$ \\
$|\mathfrak{n}^+_{\Theta}|$ & $=\frac{1}{2}(q^2-q)+pq-q^2$ \\
$|\mathfrak{p}_\Theta| = |\mathfrak{m}| +|\mathfrak{a}|+|\mathfrak{n}^+_{\Theta}|+|\mathfrak{n}^+({\Theta})| + |\mathfrak{n}^-(\Theta)|$ & $=\frac{1}{2} \big( p^2 + 2q^2 -p \big)$ \\
\hline
\end{tabular}
\end{center}
Hence, the codimencion of $P_\Theta $ is $\displaystyle \frac{1}{2} \big( -q^2 + 2pq -q \big)$. It is not difficult to see that for $p > q \geq 4$, this codimension is greater that $p+q-1$ and so, an $H' \leq P_\Theta \leq G$ such that dim$\big( G/H' \big)=p+q-1$ cannot exist. For $(p,q)=(4,3)$, we get that codim$(P_\Theta) = 6$ $(=p+q-1)$.

For $(p,q)=(4,4)$, the transitive $\text{SO}(4) \times \text{SO}(4)$ cases are the ones with the following $\mathcal{O}^p$'s and $\mathcal{O}^q$'s:
\begin{align*}
   (\text{i}) \,  \mathcal{O}^p \simeq \text{SO}(4) \big/ \text{SU}_2 \quad & \text{ and } \quad \mathcal{O}^q \simeq \text{SO}(4) \big/ \left( \text{SO}(2) \times \text{SO}(2) \right) \\
   (\text{ii}) \,  \mathcal{O}^p \simeq \text{SO}(4) \big/ \text{SO}(3) \quad & \text{ and } \quad \mathcal{O}^q \simeq \text{SO}(4) \big/ \left( \text{SO}(2) \times \text{SO}(2) \right) \\
  (\text{iii}) \,    \mathcal{O}^p \simeq \text{SO}(4) \big/ \text{SO}(2) \quad & \text{ and } \quad \mathcal{O}^q \simeq \text{SO}(4) \big/  \text{U}_2  \\
\end{align*}
For $(p,q)=(4,3)$, the cases are:
\begin{alignat*}{2}
  (\text{iv}) & \, \mathcal{O}^p \simeq \text{SO}(4) \big/ \left( \text{SO}(2) \times \text{SO}(2) \right) \quad && \text{ and } \quad \mathcal{O}^q \simeq \text{SO}(3) \big/  \text{SO}(2) \\
  (\text{v}) & \, \mathcal{O}^p \simeq \text{SO}(4) \big/ \text{SU}_3 \quad && \text{ and } \quad \mathcal{O}^q \simeq \text{SO}(4) \big/  \text{SO}(3) \\
\end{alignat*}
For $(p,q)=(3,3)$, the only case is:
\[ \mathcal{O}^p \simeq \text{SO}(3) \big/ \text{SO}(2) \quad  \text{ and } \quad \mathcal{O}^q \simeq \text{SO}(3)  \]

In the following, we assume that the action of $\text{SO}(p) \times \text{SO}(q)$ on $M$ extends to an action of $G = \text{SO}^\circ(p,q)$ and that $H' \leq P_\Theta$; recall $H' = \text{Stab}_G(x)$.

For $(p,q)=(4,4)$, recall that $P_\Theta$ is the parabolic subgroup of $G$ that corresponds to the subset $ \Theta = \{ f_1-f_2, \cdots , f_{3} - f_4 \}$. Changing the quadratic form on $\mathbb{R}^8$ from 
\[ -x_1^2 -x_2^2 -x_3^2 - x_4^2 + x_5^2 + x_6^2 + -x_7^2+ x_8^2  \]
to 
\[ 2x_1x_8 + 2x_2x_7 + 2x_3x_6+2x_4x_5\]
 and computing the Lie algrebra of $P_\Theta$ in this new representation, we can see that $P_\Theta$ leaves invariant the subspace spanned by the first 4 basis vectors, which is a maximal isotropic 4-dimensional subspace. This computation is rather straightforward, but also prolix and so it is omitted. Consequently, if $P_\Theta = M_\Theta A_\Theta N_\Theta$ is the Langlands decomposition, we have that $M_\Theta \simeq \text{SL}_4(\mathbb{R})$. If the action of $K$ extends to $G$ and $H' \leq P_\Theta$, by Witte's theorem, $N_\Theta \leq H'$.\\
$\bullet$ If $A_\Theta \leq H'$, then again by Witte's theorem, there exists a closed, connected and \textit{normal} subgroup $Y$ of $M_\Theta$, such that $(H')^\circ = YA_\Theta N_\Theta$. Counting dimensions, $Y$ needs to be of codimension 1 in $M_\Theta$. But, $M_\Theta \simeq \text{SL}_4(\mathbb{R})$ is simple, which gives a contradiction. \\
$\bullet$ If $A_\Theta \nleq H'$, then by counting dimensions, $(H')^\circ = M_\Theta^\circ N_\Theta$. Now, we have $\text{SU}_2 \times \text{SO}(2) \times \text{SO}(2) \leq (H')^\circ$. Since $\text{SU}_2 \times \text{SO}(2) \times \text{SO}(2)$ is compact, it is necessarily contained
in a maximal compact subgroup of $M_\Theta^\circ$; recall that $N_\Theta$ is unipotent. Hence, a conjugate of $\text{SU}_2 \times \text{SO}(2) \times \text{SO}(2)$ is a subgroup of $\text{SO}(4)$. But, $\text{SU}_2 \times \text{SO}(2) \times \text{SO}(2)$ has rank 3, while $\text{SO}(4)$ has rank 2, which gives us a contradiction for case (i). A similar argument gives a contradiction for case (ii).

For $(p,q)=(4,3)$, in the case of (iv) : $\text{SU}_2 \times \{1\} \leq H'$. By a similar analysis like above, 
we can see that $M_\Theta^\circ \simeq \text{SL}_3(\mathbb{R})$. \\
$\bullet$ If $A_\Theta \leq H'$, like above we get a contradiction by the simplicity of $M_\Theta^\circ$.\\
$\bullet$ If $A_\Theta \nleq H'$, then $\text{SU}_2 \leq (H')^\circ = M_\Theta^\circ N_\Theta$ is contained in a maximal compact of $M_\Theta^\circ \simeq \text{SL}_3(\mathbb{R})$. Hence, a conjugate of $\text{SU}_2$ is a subgroup of $\text{SO}(3)$. However, both $\text{SU}_2$ and $\text{SO}(3)$ are 3-dimensional and connected, hence this conjugation gives an isomorphism between them, which is a contradiction. A similar argument gives a contradiction for case (v).

For $(p,q)=(3,3)$, we have dim$(G)=15$ and dim$(H')=10$. Like before, we can see that $M_\Theta \simeq \text{SL}_3(\mathbb{R})$. We write again $P_\Theta = M_\Theta A_\Theta N_\Theta$ for the Langlands decomposition. Recall, $N_\Theta \leq H'$. Witte's theorem gives a closed, connected and \textit{normal} $Y \trianglelefteq M_\Theta^\circ$ inside $(H')^\circ$. Since $M_\Theta \simeq \text{SL}_3(\mathbb{R})$ is simple, $Y$ is either all of $M_\Theta^\circ$ or trivial. In the former case, the dimension of $H'$ is at least 11, while in the latter it is at most 4. Both contradict the fact that dim$(H')=10$.

\subsection{Non-extendable non-transitive cases} \label{non-ext non-trnstv axns}

Here, we treat cases where the action of $K$ does not extend to $G$, but we do not have an immediate contradiction by counting dimensions.

\begin{lemma} \label{U2 U2}
    Suppose $\textup{SO}(4)\times \textup{SO}(4)$ acts on a manifold $M$ of dimension 7. Let $x \in M$ and denote $\mathcal{O}^p$ the orbit of $x$ with respect to $\textup{SO}(4)\times \{ 1 \}$ and $\mathcal{O}^q$ the orbit of $x$ with respect to $\{1 \} \times \textup{SO}(4)$. Suppose 
\[ \mathcal{O}^p \simeq \textup{SO}(4) \big/ \textup{U}_2 \quad \textup{ and } \quad \mathcal{O}^q \simeq \textup{SO}(4) \big/ \textup{U}_2 \]
Then, the action is not locally effective. In particular, it cannot extend to an action of $\textup{SO}^\circ(4,4)$.
\end{lemma}

\begin{proof}
    Assume it is locally effective. We will get a contradiction by looking at the isotropy representation of $\text{U}_2 \times \text{U}_2$ at $x$. The subspace $ \text{T}_x \mathcal{O}^p \oplus \text{T}_x \mathcal{O}^q$ is an invariant subspace of $\text{T}_xM$ with respect to the isotropy representation. By introducing a $\text{U}_2 \times \text{U}_2$ invariant metric on $M$, we can take the orthogonal complement of the aforementioned subspace and hence obtain a decomposition $\text{T}_xM = \text{T}_x \mathcal{O}^p \oplus \text{T}_x \mathcal{O}^q \oplus W$, where $W \simeq \mathbb{R}^3$ is also invariant under the action of $\text{U}_2 \times \text{U}_2$. Now, since $\text{U}_2 \times \text{U}_2 $ acts by isometries, the isotropy representation gives us a map from $\text{U}_2 \times \text{U}_2$ into block matrices of the form $ \displaystyle \begin{bmatrix}
    A&& \\ &B& \\ &&C
\end{bmatrix}$, where $A,B \in \text{O}_2$ and $C \in \text{O}_3$. However, dim$(\text{U}_2 \times \text{U}_2)=8$, while dim$(\text{O}_2\times \text{O}_2 \times \text{O}_3 ) = 5$. This means that the map above must have nontrivial kernel, i.e. there exists a subgroup $H$ of $\text{U}_2 \times \text{U}_2$ with dimension at least 1, that acts trivially on $\text{T}_xM$. This contradicts the assumption that $\text{SO}(4)\times \text{SO}(4)$ acts locally effectively.
\end{proof}

\begin{lemma} \label{U2 case1}
    Suppose $\textup{SO}(4)\times \textup{SO}(q)$, $q \geq 3$, acts on a manifold $M$ of dimension $q+3$. Let $x \in M$ and denote $\mathcal{O}^p$ the orbit of $x$ with respect to $\textup{SO}(4)\times \{ 1 \}$ and $\mathcal{O}^q$ the orbit of $x$ with respect to $\{1 \} \times \textup{SO}(q)$. Suppose 
\[ \mathcal{O}^p \simeq \textup{SO}(4) \big/ \textup{U}_2  \]
and that $\textup{SO}(q)$ fixes $x$. Then, the action is not locally effective. In particular, it cannot extend to an action of $\textup{SO}^\circ(4,q)$.
\end{lemma}

\begin{proof}
Consider a $\text{U}_2 \times \text{SO}(q)$ invariant metric and decompose $\text{T}_xM$ as $\text{T}_xM = \text{T}_x \mathcal{O}^p \oplus W$, where $W \simeq \mathbb{R}^{q+1}$ is invariant under the action of $\text{U}_2 \times \text{SO}(q)$; here, $\text{T}_x\mathcal{O}^q$ is trivial. Now, $\text{SO}(q)$ acts on $W$. If that action is trivial, then we get that $\text{SO}(q)$ acts trivially on $\text{T}_xM$ and this contradicts the local effectiveness of the action. Assume momentarily that $q\neq 4$. Then, the action of $\text{SO}(q)$ on $W$ decomposed as $W \simeq \mathbb{R}^q \oplus \mathbb{R}^1$, with $\mathbb{R}^q$ being the standard representation of $\text{SO}(q)$ and $\mathbb{R}^1$ being the trivial one. Since $\text{U}_2$ preserves $W$ and sends $\text{SO}(q)$-invariant subspaces to $\text{SO}(q)$-invariant subspaces because $\text{U}_2$ and $\text{SO}(q)$ commute, $\text{U}_2$ preserves $\mathbb{R}^q$ and $\mathbb{R}^1$. Therefore, we get a map from $\text{U}_2 \times \text{SO}(q)$ to block matrices of the form $ \displaystyle \begin{bmatrix}
    A && \\ &B& \\ &&1
\end{bmatrix}$, where $A \in \text{O}_2$ and $B \in \text{O}_q$. But dim$(\text{U}_2 \times \text{SO}(q))=4 + \frac{1}{2}q(q-1)$, while dim$(\text{O}_2 \times \text{O}_q)=1 + \frac{1}{2}q(q-1)$. Therefore we get a subgroup of $\text{U}_2 \times \text{SO}(q)$ that acts trivially on $\text{T}_xM$ and that contradicts the local effectiveness of our action. If $q \neq 4$, then the action of $\text{SO}(4)$ on $\mathbb{R}^5$ could decompose differently, but that would only result in smaller blocks in the place of ``$B$'' above. Counting dimensions, we would arrive at a similar contradiction.
    \end{proof}

\begin{lemma} \label{U2 case 2}
    Suppose $\textup{SO}(4)\times \textup{SO}(q)$, $q \geq 3$, acts on a manifold $M$ of dimension $q+3$. Let $x \in M$ and denote $\mathcal{O}^p$ the orbit of $x$ with respect to $\textup{SO}(4)\times \{ I_q \}$ and $\mathcal{O}^q$ the orbit of $x$ with respect to $\{I_4 \} \times \textup{SO}(q)$. Suppose 
\[ \mathcal{O}^p \simeq \textup{SO}(4) \big/ \textup{U}_2 \quad \textup{and} \quad \textup{dim}(M) -\textup{dim}(\mathcal{O}^q) - \textup{dim}(\mathcal{O}^p) \leq 2\]
Then, the action is not locally effective. In particular, it cannot extend to an action of $\textup{SO}^\circ(4,q)$.
\end{lemma}

\begin{proof}
Indeed, consider $\text{SU}_2 \leq \text{U}_2$. Then, $\text{SU}_2$ fixes $x$. Looking at the isotropy representation, $\text{SU}_2$ acts trivially on $\text{T}_x\mathcal{O}^p$ since it is 2-dimensional, see \cite{ITZKOWITZ1991285}. It also acts trivially on $\text{T}_x\mathcal{O}^q$. Finally, introducing a $\text{U}_2$ invariant metric on $M$ and taking the orthogonal complement of $\text{T}_x\mathcal{O}^p \oplus \text{T}_x\mathcal{O}^q$, which will also be $\text{SU}_2$-invariant, $\text{SU}_2$ acts trivially on it also by our assumption on the dimensions. Therefore, $\text{SU}_2$ acts trivially on $\text{T}_xM$ and so the action is not locally effective.
\end{proof}

\begin{lemma}
Suppose $\textup{SO}(9)\times \textup{SO}(7)$ acts on a manifold $M$ of dimension $15$. Let $x \in M$ and denote $\mathcal{O}^p$ the orbit of $x$ with respect to $\textup{SO}(9)\times \{ I_7 \}$ and $\mathcal{O}^q$ the orbit of $x$ with respect to $\{I_9 \} \times \textup{SO}(7)$. Suppose 
\[ \mathcal{O}^p \simeq \textup{SO}(9) \big/ \textup{Spin}(7) \]
Then, the action is not locally effective. In particular, it does not extend to an action of $\textup{SO}^\circ(9,7)$.
    \end{lemma}

\begin{proof}
        Indeed, we have that dim$\left( \text{SO}(9) \big/ \text{Spin}(7) \right) = 15$. Let $H_7' \leq \text{SO}(7)$ be the isotropy group of $x$, for the restricted action of $\{I_9\} \times \text{SO}(7)$. Then, $H_7'$ acts trivially on $\text{T}_x\mathcal{O}^p$ which is all of $\text{T}_xM$. If dim$(H_7')\geq 1$ then the action is not locally effective. If dim$(H_7')=0$, then dim$(\mathcal{O}^q) =$ dim$(\text{SO}(7)) = 21 > 15 =$ dim$(M)$, a contradiction.
   \end{proof}

\begin{lemma}
        Suppose that $\textup{SO}(p)\times \textup{SO}(q)$ acts on $M$. Let $x \in M$ denote $\mathcal{O}^p$ the orbit of $x$ with respect to $\textup{SO}(p)\times \{ I_q \}$ and $\mathcal{O}^q$ the orbit of $x$ with respect to $\{I_p \} \times \textup{SO}(q)$. If
        \[ \textup{dim} \mathcal{O}^p = p-1 \quad \text{and}  \quad \textup{dim} \mathcal{O}^q = q-1\]
        Let 
        \[ \mathcal{O}^p \simeq \textup{SO}(p) \Big/ H_p' \quad \text{and} \quad \mathcal{O}^q \simeq \textup{SO}(q) \Big/ H_q'\]
        and suppose the isotropy representations of $H_p'$ and $H_q'$ are irreducible, and $-I_p \in H_p'$, where $I_p$ is the $p\times p$ identity matrix. Assume that the action extends to $\textup{SO}^\circ(p,q)$. Then, the $\textup{SO}^\circ(p,q)$-orbit of $x$, denoted $\mathcal{O}$, cannot be an open orbit.
    \end{lemma}
\begin{proof}
    We can assume that $M$ is endowed with an $\text{SO}(p)\times \text{SO}(q)$-invariant metric. Looking at the isotropy representation of $H_p' \times H_q'$, there exists an invariant line in $\text{T}_xM$ because of the assumptions on the dimensions of the orbits, and $H_p' \times H_q'$ must act trivially on this line.  
    
     The Lie algebra $\mathfrak{so}(p,q)$ comprises matrices of the following form $\begin{bmatrix}
        X_1 & X_2 \\
        X_3 & X_4
    \end{bmatrix}$, where $X_1 \in \mathfrak{so}_p$, $X_4 \in \mathfrak{so}_q$, $X_3 = X_2^\text{T}$ and $X_2$ is a $p \times q$ matrix. We identify $H_p' \times H_q'$ with $\displaystyle \left\{ \begin{bmatrix}
        h_1 &\\&h_2
    \end{bmatrix} : h_1 \in H_p', \, h_2 \in H_q' \right\}$. Then, if $h_0 = (h_1,h_2) \in H_p' \times H_q'$ and  
    $X= \begin{bmatrix}
        X_1 & X_2 \\
        X_2^\text{T} & X_4
    \end{bmatrix} \in \mathfrak{so}(p,q)$, $\text{Ad}_{h_0} \left( X \right) = \begin{bmatrix}
        h_1X_1h_1^\text{T} & h_1X_2h_2^\text{T} \\
        h_2X_2^\text{T}h_1^\text{T} & h_2X_4h_2^\text{T}
    \end{bmatrix}$. Now, assume that $\mathcal{O}$ is an open orbit. Then, $\mathcal{O} \simeq \text{SO}^\circ(p,q) \Big/ G_x$, where $G_x$ is the isotropy group of $x$ with respect to the $\textup{SO}^\circ(p,q)$-action. If $\mathfrak{g}_x$ is the isotropy algebra at $x$, then $\text{T}_xM$ can be identified with $\mathfrak{so}(p,q) \Big/ \mathfrak{g}_x$ and the isotropy representation of $G_x$ with the adjoint representation mod $\mathfrak{g}_x$ (see \cite[Corollary 10.2.13]{HilgertNeeb}). Let ${v} \in \text{T}_xM$ be a nonzero vector on which $H_p' \times H_q'$ acts trivially. There exists an $X^{v} \in \mathfrak{so}(p,q)$ such that $X^v \, \text{mod} \mathfrak{g}_x \equiv v$ under the identification mentioned above. Write $X^v = \begin{bmatrix}
        X^v_1 & X^v_2 \\
        (X^v_2)^\text{T} & X^v_4
    \end{bmatrix}$. Since the isotropy representations of $H_p'$ and $H_q'$ are irreducible, we can assume that $X^v_1 = X^v_4 = 0$. Since $v$ is non zero, $X^v_2 \neq 0$ and $\text{Ad}_{h_0} \left( X^v \right) \equiv X^v \, \text{mod} \mathfrak{g}_x$, for $h_o \in H_p' \times H_q'$. But, $h_0 = (-I_p, I_q) \in H_p' \times H_q'$ and $\text{Ad}_{h_0} \left( X^v \right) = - X^v$ which implies $- X^v \equiv X^v \, \text{mod} \mathfrak{g}_x$. But this is a contradiction, since $X^v \notin \mathfrak{g}_x$.
\end{proof}

 Now, suppose $K = \text{SO}(p) \times \text{SO}(q)$ acts on a manifold $M$, $x \in M$ and either $p=8$ and $H_p' \simeq \text{Spin}(7)$ or $p=4$ and $H_p' \simeq \text{SU}_2$, where $H_p'$ is like in the lemma above. Furthermore, assume that the $\text{SO}(q)$-orbit of $x$ is any of the first three in the bottom part of Table  \ref{table:1} and that the $K$ action extends to a $G$ action. If $\text{SO}(q)$ fixes $x$, note that the isotropy representation of $\text{SO}(q)$ on $\text{T}_x \mathcal{O}^q$ has to be the standard irreducible one; so, we can always find a point near $x$ in the same $G$-orbit that is fixed by neither $\text{SO}(p)$ nor $\text{SO}(q)$. Then, $\mathcal{O}^p$ and $\mathcal{O}^q$ for this new point are some from Table \ref{table:1}. So we can assume that $x$ is not fixed by $\text{SO}(q)$. Since the respective $-I_p$ is in both $\text{SU}_2$ and Spin$(7)$, see \cite{Varadarajan}, by the lemma above the $G$-orbit of $x$ cannot be open in $M$. Therefore, it is closed and hence compact. If $G_x$ is the isotropy group of $x$ with respect to the $G$ action, $G_x$ is cocompact and hence we can apply Witte's result and by similar reasoning as in Section \ref{2nd kind parabolics} we can show that this is impossible. Therefore, such $K$ actions do not extend to $G$.

\subsection{Conclusion} 

The following proposition now follows from the results of this section:
\begin{prop} \label{concl for SOpxSOq orbit types}
    Let $p,q \geq 3$ and suppose $\textup{SO}^\circ(p,q)$ acts analytically on a manifold $M$. Then, the restricted $K$ action on $M$ is not transitive and with respect to the restricted action of $\textup{SO}(p) \times \{ I \} \leq K$, respectively $\{ I \} \times \textup{SO}(q) \leq K$, the orbit type of a point $x \in M$ is either $\textup{SO}(p) \big/ \textup{SO}(p-1)$ or $x$ is fixed by $\textup{SO}(p)$, respectively $\textup{SO}(q) \big/ \textup{SO}(q-1)$ or fixed by $\textup{SO}(q)$. 
\end{prop}

 In particular, we have
\begin{corollary} \label{fix pnt set of H nonempty}
    Let $p,q \geq 3$ and suppose $\textup{SO}^\circ(p,q)$ acts analytically on $M$. The fixed point set of $H = \textup{SO}(p-1) \times \textup{SO}(q-1)$ is non-empty and 1-dimensional.
\end{corollary}
 This follows from the above Proposition and Lemma \ref{lemma Uch}. Of course, all the components of the aforementioned fixed point set are isomorphic to $\text{S}^1$.
\section{Classification of analytic $\text{SO}^\circ(p,q)$, $p,q\geq3$ actions} \label{SO(p,q) axns}

In this section, we assume $G= \text{SO}^\circ(p,q)$ acts analytically and not trivially on a manifold $M$ of dimension dim$(M)=p+q-1$. As we will see, the presence or absence of fixed points for the subgroups $\text{SO}(p)$ and $\text{SO}(q)$ plays an important role in studying these actions and has consequences on the topology and geometry of $M$. The following theorem is an improved version of Theorem \ref{general thm} in the Introduction.

\begin{theorem} \label{final result SOpq}
    Suppose $\textup{SO}^\circ(p,q)$, $p,q \geq 3$, acts analytically on a  closed, connected manifold $M$ of dimension $p+q-1$. Consider $\textup{SO}(p) \simeq \textup{SO}(p) \times \{1\} \leq \textup{SO}(p) \times \textup{SO}(q) \leq \textup{SO}^\circ(p,q)$ and $\textup{SO}(q)$ similarly.
    \begin{itemize}
        \item Suppose that there exists a point in $M$ that is fixed by $\textup{SO}(p)$ and that there exist no points fixed by $\textup{SO}(q)$. Then $M$ is equivariantly covered by $\textup{S}^p \times \textup{S}^{q-1}$, where the action of $\textup{SO}^\circ(p,q)$ on $\textup{S}^p \times \textup{S}^{q-1}$ is one from the basic construction, see Section \ref{basic constr}.
        \item Suppose that there exists a point in $M$ that is fixed by $\textup{SO}(q)$ and that there exist no points fixed by $\textup{SO}(p)$. Then $M$ is equivariantly covered by $\textup{S}^{p-1} \times \textup{S}^{q}$, where the action of $\textup{SO}^\circ(p,q)$ on $\textup{S}^{p-1} \times \textup{S}^q$ is one from the basic construction, see Section \ref{basic constr}.
        \item Suppose both $\textup{SO}(p)$ and $\textup{SO}(q)$ have fixed points in $M$.  Then, $M$ is equivariantly covered by $\textup{S}^{p+q-1}$, where the action of $\textup{SO}^\circ(p,q)$ on $\textup{S}^{p+q-1}$ is one from the Uchida construction corresponding to a basic $(\textup{J}_1,\textup{J}_2)$-flow, see Section \ref{Uchida constr section}.
        \item Suppose that neither $\textup{SO}(p)$ nor $\textup{SO}(q)$ have a fixed point. Then, the action is equivariantly covered by $G \times_P S$ with the standard left $G$ action, where $P \leq G$ is a maximal parabolic subgroup isomorphic to the stabiliser of an isotropic line in the standard representation of $\textup{SO}^\circ(p,q)$ on $\mathbb{R}^{p+q}$. If $P=M_P A_P N_P$ is the Langlands decomposition of $P$, $P$ acts on $\textup{S}^1$ by a flow via $A_P$, see \textup{(\ref{axn of P on S1})}.
    \end{itemize}
\end{theorem}
 The proof of the first two cases is in Section \ref{axn with SOp fixed pnt}, the third case is shown in Section \ref{axns with both fxd pts}, while the last case is shown in Section \ref{axns w nullcone}.

\begin{lemma}
    $G$ does not have a fixed point.
\end{lemma}
 Indeed, we actually show something stronger:
\begin{lemma} \label{K fixed pnt}
    $K= \text{SO}(p) \times \text{SO}(q)$ does not have a fixed point.
\end{lemma}
\begin{proof}
    Assume it does, and let $x \in M$ be a fixed point of $K$. Note that, since $\text{SO}^\circ(p,q)$ is simple, its action is locally effective, and so the $K$ action is also locally effective. Suppose $p \geq q$ and that $p \neq 4$. Consider the isotropy representation of $\text{SO}(p)$ at $x$. If $\text{SO}(p)$ acted trivially on $\text{T}_xM \simeq \mathbb{R}^{p+q-1}$, then it would act trivially on a neighbourhood of $x$ in $M$ and so the action would not be locally effective. Therefore, we have a non-trivial representation of $\text{SO}(p)$ on $\mathbb{R}^{p+q-1}$. The only possible decomposition into irreducibles is $\mathbb{R}^{p+q-1} = \mathbb{R}^p \oplus \mathbb{R}^{q-1}$, where the representation on the first summand is the standard one and on the second it is trivial. Since $\text{SO}(p)$ and $\text{SO}(q)$ commute, the two summands are also invariant by $\text{SO}(q)$ and hence by $K$. Then, if we introduce a $K$-invariant metric on $M$, the isotropy representation of $K$ at $x$ gives a homomorphism $\rho : K \rightarrow \text{SO}(p) \times \text{SO}(q-1)$. Counting dimensions, we see that the kernel of $\rho$ is not 0-dimensional, which contradicts the local effectiveness of the action. If $p=4$ and the decomposition of $\mathbb{R}^{4+q-1}$ into irreducibles is not like above, then it would have to be $\mathbb{R}^{3} \oplus \mathbb{R}^3 \oplus \mathbb{R}$ or $\mathbb{R}^3 \oplus \mathbb{R}^3$, depending on whether $q=4$ or 3 respectively. In these cases, $\rho : \text{SO}(4) \times \text{SO}(4) \rightarrow \text{SO}(3) \times \text{SO}(3)$ or $\rho : \text{SO}(4) \times \text{SO}(3) \rightarrow \text{SO}(3) \times \text{SO}(3)$ respectively. Counting dimensions and considering the kernel of $\rho$, we arrive at a contradiction like before.
\end{proof}
    
 Recall $H = \text{SO}(p-1) \times \text{SO}(q-1)$ and $\mathcal{F} = \text{fixed point set of }H$. Now, by Corollary \ref{fix pnt set of H nonempty}, $\mathcal{F} \neq \emptyset $ and $ \mathcal{F} = \bigcup_{i}S_i$ , where the $S_i$'s are the connected components of $\mathcal{F}$ and with each $S_i$ diffeomorphic to $\text{S}^1$. Let $\Sigma$ be one of these components. By Lemma \ref{K fixed pnt} and Lemma \ref{lemma Uch}, we can define an analytic function $f: \Sigma \rightarrow \mathbb{R}\text{P}^1$, see Remark \ref{rmk definition of f}. We also get an analytic flow $\Phi_\theta$ or equivalently, a vector field $X^\Phi$ on $\Sigma$ by the action of $\mathcal{M}(p,q)$.

\begin{lemma} \label{f =/ [1:1] lemma}
    If there exists a point $\zeta \in \Sigma$ such that $f(\zeta) = [a:b] \neq [\pm 1 , 1]$, then there exists a point $\zeta ' \in \Sigma$ which is a fixed point for $\text{SO}(p)$ or $\text{SO}(q)$.
\end{lemma}
\begin{proof}
    Set $f(z) = [a:b] \neq [\pm1:1]$. Now, $\Phi_\theta$ and $f$ satisfy property (iii) from the Uchida conditions in Remark \ref{Uchida conditions}, see \cite{Uchida}. Namely, $f \left( \Phi_\theta (\zeta) \right) = \left[ a \, \text{cosh}(\theta) + b \, \text{sinh}(\theta) : a \, \text{sinh}(\theta) + b \, \text{cosh}(\theta) \right] $. Letting $\theta \rightarrow + \infty$, it is immediate that there exists a point $\zeta ' \in \Sigma$ such that $f( \zeta ') = [0:1]$ or $[1:0]$. Then, $\zeta '$ is a fixed point for $\text{SO}(p)$ or $\text{SO}(q)$ respectively, by the definition of $f$.
\end{proof}
\begin{remark} \label{nullcone <-> no fixed pts for SOp or SOq}
    If we assume that neither $\text{SO}(p)$ nor $\text{SO}(q)$ has a fixed point in $M$ and in particular in $\Sigma$, then $f(\zeta) = [\pm1:1]$ for all $\zeta \in \Sigma$. Then, since $f$ is analytic, it is constantly equal to $[1:1]$ or $[-1:1]$ on all of $\Sigma$. We note that in that case, the orbit of a point in $\Sigma$ at which $X^\Phi$ is not zero, is analytically isomorphic to a component of a nullcone in $\mathbb{R}^{p+q}$ equipped with a scalar form of signature $(p,q)$. We call these, \textit{nullcone orbits}. Therefore, the existence of a nullcone orbit forces all the orbits to either be nullcone orbits or diffeomorphic to $\text{S}^{p-1} \times \text{S}^{q-1}$ as the next section shows and implies that there are no fixed points for $\text{SO}(p)$ or $\text{SO}(q)$.
\end{remark}

\subsection{$\text{SO}^\circ(p,q)$ actions with neither $\text{SO}(p)$ nor $\text{SO}(q)$ fixed points} \label{axns w nullcone}

Let $\mathcal{F}:=$ Fix$\left( \text{SO}(p-1) \times \text{SO}(q-1) \right)$. The connected components of $\mathcal{F}$ are diffeomorphic to circles. Let $\Sigma$ be one component of $\mathcal{F}$. On $\Sigma$, by Lemma \ref{lemma Uch}, we can define an analytic function $f: \Sigma \rightarrow \mathbb{R}\text{P}^1$, see Remark \ref{rmk definition of f}. Suppose neither $\text{SO}(p)$ nor $\text{SO}(q)$ has a fixed point in $M$. Then, by Remark \ref{nullcone <-> no fixed pts for SOp or SOq}, $f \equiv [1:1]$ or $[-1:1]$. Without loss of generality, we assume $f \equiv [1:1]$. Then, for every $\zeta \in \Sigma$, the isotropy group of $\zeta$ is contained in or equal to a maximal parabolic subgroup $P$ of $G$ of the first kind, see Section \ref{firstkindparab}.

Let $\text{S}^1$ be the usual unit circle in $\mathbb{R}^2$. Identify $\Sigma$ and $\text{S}^1$ by an analytic isomorphism $\psi:  \text{S}^1 \rightarrow \Sigma$. Recall that $\mathcal{M}(p,q)$ is the subgroup of $G$ consisting of the matrices $m(\theta) := \begin{bmatrix}
    \text{cosh}(\theta) && \text{sinh}(\theta) & \\ & I_{p-1} & & \\ \text{sinh}(\theta) && \text{cosh}(\theta) & \\
    &&& I_{q-1}
\end{bmatrix}$
for $\theta \in \mathbb{R}$. Now, $\Sigma$ is invariant under the action of $\mathcal{M}(p,q)$, and on $\Sigma$ we have a flow $\Phi^\Sigma_{\theta}(\zeta) = m(\theta) \star \zeta$, where $\star$ denotes the action of $G$. Define a flow on $\text{S}^1$ by
\begin{equation*}
    \Phi_{\theta}(z) := \psi^{-1} \big( m(\theta) \star \psi(z) \big)
\end{equation*}
Let $\pi : P \rightarrow (\mathbb{R},+)$ be the following homomorphism: Let $P=M_P A_P N_P$ be the Langlands decomposition of $P$. Here, $A_P = \mathcal{M}(p,q)$ and $N_P = U(z)$, for any $z \in \Sigma$ see (\ref{eq:1}). Define $\pi$ so that it picks out the ``$A_P$-part''. Namely, if $p \in P$, we write $p=p_{M_P} \cdot p_{A_P} \cdot p_{N_P}$ according to the Langlands decomposition above. Then, $p_{A_P} = m(\theta)$ for some $m(\theta) \in \mathcal{M}(p,q)$ and we define $\pi(p) : = \theta$. It is easy to see that $\pi$ is an analytic homomorphism between $P$ and $(\mathbb{R},+)$; recall that for a parabolic subgroup, $N_P \trianglelefteq P$, and $M_P$ and $A_P$ commute.

Now, we define the following action of $P$ on $\text{S}^1$:
\begin{align} \label{axn of P on S1}
\begin{split}
    P \times \text{S}^1 & \rightarrow \text{S}^1 \\
    (p,z) & \mapsto \Phi_{\pi(p)}(z)
\end{split}
\end{align}
and consequently we get an action of $G$ on the associated bundle $G \times_P \text{S}^1$ where $G$ acts by left multiplication on the first factor. It is easy to see that this action is well defined. Note also, that for $p \in P$ and $\zeta \in \Sigma$,
\begin{equation} \label{p.z = m(theta).z}
    p \star \zeta = m(\pi(p)) \star \zeta
\end{equation}
\begin{remark}
Suppose we are given two different actions of $P$ on $\text{S}^1$ via flows $\Phi^1$ and $\Phi^2$ like above. If the vector fields that these flows generate are isomorphic, then it is immediate that the two actions of $G$ on $G \times_P \text{S}^1$ are also isomorphic. Since the analytic vector fields on the circle are classified in \cite{HITCHIN1991359}, such analytic $G$ actions on $G \times_P \text{S}^1$ are also classified via the corresponding flows $\Phi$.
\end{remark}
Now, let $\Psi$ be the analytic map $\Psi : G \times_P \text{S}^1 \rightarrow M$ defined by $[g,z] \mapsto g \star \psi(z)$, where $[g,z]$ denotes the equivalence class of $(g,z)$ in $G \times_P \text{S}^1$. In the following, let $\zeta = \psi(z)$ for $z\in \Sigma$.

$\bullet$ \, $\Psi$ is well defined: \\
A different representative of the class $[g,z]$ is of the form $[pg^{-1},p \star z]$. We have, $\Psi ( [gp^{-1},p \star z]) = (gp^{-1}) \star \psi (p \star z)$. But, $p \cdot z = \Phi_{\pi(p)}(z) = \psi^{-1}(  m(\pi(p)) \star \zeta) = \psi^{-1} (p \star \zeta) $, and so $\psi (p \star z) = p \star \zeta$
\vspace{10pt}

$\bullet$ \, $\Psi$ is $G$-equivariant: \\
We have $G=KP$; this can be seen from (\ref{eq:1}). Let $[\tilde{g},z] \in G \times_P \Sigma$ and $g \in G$, and write $g \tilde{g} = kp$, according to the above decomposition. We have
\begin{align*}
    \Psi (g [\tilde{g},z]) &= \Psi ([g \tilde{g},z]) 
     = \Psi ([kp,z])
     = \Psi([k,p \star z]) 
     = \Psi( [k , \Phi_{\pi(p)}(z)]) \\
    & = k \star \psi(  \Phi_{\pi(p)}(z))
     = k \star ( m(\pi(p)) \star \zeta) \quad \text{(by the definition of } \Phi \text{ and (\ref{p.z = m(theta).z}) )} \\
    & = k \star (p \star \zeta )
     = g \tilde{g} \star \zeta 
     = g \star \Psi([\tilde{g},z])
\end{align*}

$\bullet$ \, $\Psi$ is a local analytic isomorphism: \\
By the equivariancy, it suffices to look at a point of the form $[e,z]$ for some $z \in \text{S}^1$. By using $(e,z) \in G \times \text{S}^1$ as a representative for $[e,z]$, we have an identification $\text{T}_{[e,z]}(G\times_P\text{S}^1) \simeq \mathfrak{g}/\mathfrak{p} \oplus \text{T}_z\text{S}^1$, where $\mathfrak{g}$ and $\mathfrak{p}$ are the Lie algebras of $G$ and $P$ respectively. Now, for an element $0+v$, for $v \in$ T$_z\text{S}^1$, we can find a curve of the form $\gamma : t \mapsto [e,z_t]$, such that $\gamma (0)=[e,z]$ and $\gamma ' (0)=v$. Then, since $\displaystyle \Psi \big|_{\{e\}\times_P\text{S}^1} = \psi$, we have that $\text{D}\Psi(v)=$ $\text{D}\psi(v) \in$ T$_{\psi(z)}\Sigma$ and $\text{D}\psi(v) \neq 0$ since $\psi$ is an isomorphism.

On the other hand, for an element $w+0$, for $w \in \mathfrak{g}/\mathfrak{p}$, we can find a representative $\tilde{w} = \tilde{w}_1 + \tilde{w}_2 \in \mathfrak{g}$ of $w$ such that $\tilde{w}_1 \in \mathfrak{so}_p$ and $\tilde{w}_2 \in \mathfrak{so}_q$. Consider the curves of $G\times_P\text{S}^1$, $c : t \mapsto [e^{t\tilde{w}_i},z]$, for $i=1,2$. Then, $c(0)=[e,z]$ and $c'(0)=w_i$, where $w_i$ is the class of $\tilde{w}_i$ in $\mathfrak{g}/\mathfrak{p}$. and consequently $ \displaystyle \text{D}\Psi(w_i) = \frac{d}{dt} \Big|_0 (e^{t \tilde{w}_i} \star \zeta)$. But, $e^{t \tilde{w}_i} \star \zeta \in \mathcal{O}_{\text{SO}(p)}(\zeta)$ or $\mathcal{O}_{\text{SO}(q)}(\zeta)$ where $\mathcal{O}_{\text{SO}(p)}(\zeta)$ and $\mathcal{O}_{\text{SO}(q)}(\zeta)$ are the $\text{SO}(p)$ and $\text{SO}(p)$ orbits of $\zeta$ respectively. Those orbits are $(p-1)$-dimensional and $(q-1)$-dimensional respectively by Proposition \ref{concl for SOpxSOq orbit types}. Moreover, they intersect $\mathcal{F}$ transversely. By considering all possible $w+0 \in \mathfrak{g}/\mathfrak{p} \oplus \text{T}_z\text{S}^1$, we see that $p+q-1 =$ dim ( $\text{D}\Psi (\mathfrak{g} / \mathfrak{p} )$ ). As a result, because $\mathcal{O}_{\text{SO}(p)}(\zeta)$, $\mathcal{O}_{\text{SO}(p)}(\zeta)$ and $\mathcal{F}$ intersect transversely, ($\text{D}\Psi)_{[e,z]}$ is an isomorphism, and hence $\Psi$ is a local analytic isomorphism. 

$\bullet$ \, $\Psi$ is a covering map:\\
Since $\Psi$ is a local analytic isomorphism, the image of $\Psi$ is an open set. Additionally, $\Psi$'s domain is compact, and so the image is also closed. As a result, since $M$ is a closed manifold, $\Psi$ is onto $M$. Being a local diffeomorphism between compact sets, $\Psi$ is a covering map. Hence, the fourth case of Theorem \ref{final result SOpq} is proved.


\subsection{$\text{SO}^\circ(p,q)$ actions with only $\text{SO}(p)$ fixed points or only $\text{SO}(q)$ fixed points} \label{axn with SOp fixed pnt}

Suppose $G$ acts analytically on a manifold $M$ of dimension dim$(M)=p+q-1$.
In this section, we assume that $\text{SO}(p)$ has a fixed point, but $\text{SO}(q)$ does not. We are going to show that $M$ is equivariantly, analytically covered by $\text{S}^p \times \text{S}^{q-1}$ with an action from the basic construction. Recall that $H = \text{SO}(p-1) \times \text{SO}(q-1)$ and $\mathcal{F} = \text{fixed point set of }H$. Then, $\mathcal{F} \neq \emptyset $ and $\mathcal{F} = \bigcup_{i}S_i$ with each $S_i$ one dimensional, see the discussion in the beginning of this section. Let $\Sigma := \text{S}_{i_0}$ be a connected component of $\mathcal{F}$. On $\Sigma$ we get an analytic flow, $\Phi^\Sigma$, defined by the action of the one parameter subgroup $\mathcal{M}(p,q)$. Furthermore, by Lemma \ref{lemma Uch} and Remark \ref{rmk definition of f}, we get an analytic function $f'_\Sigma : \Sigma \rightarrow \mathbb{R}\text{P}^1$. The existence of an $\text{SO}(p)$ fixed point, implies that $\text{S}_{i_0}$ contains a point fixed by $\text{SO}(p)$. Indeed, otherwise $f'_\Sigma \equiv [1:1]$ or $[-1:1]$ and then, by Section \ref{axns w nullcone}, $M$ is covered by $G \times_P \text{S}^1$ and that contradicts the existence of an $\text{SO}(p)$ fixed point. We note also that since $j_1 = \begin{bmatrix}
    -I_2 & \\ & I_{p+q-2}
\end{bmatrix} \in \text{SO}(p)$ is an involution from $\Sigma$ to itself with at least one fixed point, it has exactly two fixed points. Moreover, since we also assume that $\text{SO}(q)$ does not have a fixed point, therefore $f'_\Sigma \neq [1:0]$. Let $\mathcal{U}_1 = \{ [a:b] \in \mathbb{R}\text{P}^1 : b \neq 0 \} \subseteq \mathbb{R}\text{P}^1$ and let $\chi_1 :\mathcal{U}_1 \rightarrow \mathbb{R}$ be the function defined by $[a:b] \mapsto \frac{ a}{b}$. Then, $f_\Sigma = \chi_1 \circ f'_\Sigma$. Note that since $f'_\Sigma$ and $\Phi^\Sigma$ satisfy property (iii) from the Uchida conditions in Remark \ref{Uchida conditions}, see \cite{Uchida}, we have that if $f'(z)=[a:b]$, for $z \in \Sigma$, then $|a|<|b|$. Indeed, otherwise Uchida condition (iii) implies that, taking $\theta \rightarrow + \infty$, we would get point $\tilde{z} \in \Sigma$ with $f'_\Sigma(\tilde{z})=[1:0]$, namely $\tilde{z}$ would be a fixed point of $\text{SO}(q)$, which is impossible. Hence, $\lvert f_\Sigma \rvert \leq 1$. We note that $f_\Sigma$ and $\Phi^\Sigma$ satisfy relation (A1)-(A3) and $\text{(A4)}^\prime$ of Remark \ref{(A1)-(A4) props remark}.

Now, consider the space $\text{S}^p \times \text{S}^{q-1}$, where $\text{S}^p \subseteq \mathbb{R}^{p+1}$ and $\text{S}^{q-1} \subseteq \mathbb{R}^{q}$ and we denote the standard bases of $\mathbb{R}^{p+1}$ and $\mathbb{R}^{q}$ by $\{e_1, \cdots , e_{p+1} \}$ and $\{ \epsilon_1, \cdots, \epsilon_q \}$ respectively. Let $\mathcal{S} = \left\{ (\alpha e_1 + \beta e_{p+1}, \epsilon_1) : \alpha^2 + \beta^2 =1 \right\}$. Recall from the beginning of Section \ref{actions on Spq} that we have a standard action of $K$ on $\text{S}^p \times \text{S}^{q-1}$. The action of $j_1$ on $\mathcal{S}$ is $j_1(\alpha e_1 + \beta e_{p+1}, \epsilon_1) = (- \alpha e_1 + \beta e_{p+1}, \epsilon_1)$.
\begin{lemma} \label{circle isom 1}
    There exists an analytic isomorphism \[ \psi : \Sigma \rightarrow \mathcal{S} \] such that \[ \psi \circ j_1 = j_1 \circ \psi \]
\end{lemma}
\begin{proof}
    We consider a $j_1$-invariant metric on $\Sigma$; hence, $j_1$ acts as an isometry. Let $x \in \Sigma$ be a fixed point of $j_1$. Since $j_1$ is an involution and $x$ is an isolated fixed point, we have $(Dj_1)_x = - Id$. We identify $\text{T}_x\Sigma$ with $\mathbb{R}$ and consider $\text{exp}_x : \mathbb{R} \rightarrow \Sigma$. We can assume that $\text{exp}_x$ is $2 \pi$-periodic. Since $j_1$ is an isometry, we have $\text{exp}_x \left( Dj_1 (t) \right) = j_1 \circ \text{exp}_x (t)$ for any $t \in \mathbb{R}$. Now, the map $\tilde{\psi} : \mathbb{R} \big/ 2 \pi \mathbb{Z} \rightarrow \Sigma$ defined by $t + 2 \pi \mathbb{Z} \mapsto \text{exp}_x(t)$ is analytic, and if we identify its domain with $[ - \pi , \pi ]$, we have $\tilde{\psi} : [ -\pi, \pi ] \rightarrow \Sigma$ such that 
    \begin{equation} \label{psi eq}
         \tilde{\psi}(-t) = j_1 \circ \tilde{\psi}(t)
\end{equation}
Then, we identify $[ -\pi, \pi ]$ with $\text{S}^1$ by $t \mapsto e^{i t}$, which is also analytic. Let $\rho_0: \text{S}^1 \rightarrow \mathcal{S}$ be the analytic isomorphism from (\ref{rho: S1 -> mathcal S}). Finally, set $\psi = \rho_0 \circ \tilde{\psi}^{-1}$. Equation (\ref{psi eq}) above implies $\psi \circ j_1 = j_1 \circ \psi$
\end{proof}

Using $\psi$ we can define a function $f: \mathcal{S} \rightarrow \mathbb{R}$ by $f = f_\Sigma \circ \psi^{-1}$ and a flow $\Phi : \mathbb{R} \times \mathcal{S} \rightarrow \mathcal{S}$ by $\Phi_\theta (z)  = \psi^{-1} \circ \Phi^\Sigma_\theta \left( \psi(z) \right)$. Then, $\Phi_\theta$ and $f$ are analytic on $\mathcal{S}$ and it is easy to see that they satisfy the relations (A1)-(A4) from Remark \ref{(A1)-(A4) props remark}. Therefore, $f = f_\Phi$, see Remark \ref{f_phi def remark}, and $\Phi_\theta$ is an induced basic $j_1$-flow on $\mathcal{S}$, see Lemma \ref{relations <-> basic flow}. From $\Phi_\theta$ we get an analytic action of $G$ on $\text{S}^p \times \text{S}^{q-1}$ by the basic construction.

We now define the map $F : \text{S}^p \times \text{S}^{q-1} \rightarrow M$ in the following way. Let $x \in\text{S}^p \times \text{S}^{q-1}$. There exist $(\kappa_1, \kappa_2) \in K$ and $z \in\mathcal{S}$ such that $x = (\kappa_1, \kappa_2) \star z$. Then, define
\[ F \left( (\kappa_1, \kappa_2) \star z \right) : = (\kappa_1, \kappa_2) \star \psi(z)\]
$F$ is well defined and locally an analytic isomorphism. 
\begin{lemma}
    $F$ is well defined
\end{lemma}
\begin{proof}
    Let $z = \left( \alpha e_1 + \beta e_{p+1} , \epsilon_1 \right)$ and $z' = \left( \alpha' e_1 + \beta' e_{p+1} , \epsilon_1 \right)$ be elements of $\mathcal{S}$ and $k=(\kappa_1, \kappa_2), k' = ( \kappa_1', \kappa_2') \in K$. Suppose $k \star z = k' \star z'$. Then, we have $\kappa_2^{-1}\kappa_2' \in \text{SO}(q-1)$, $\beta = \beta'$ and $\alpha = \pm \alpha'$.

If $\alpha = \alpha' \neq 0$, then $z=z'$ and $\kappa_1^{-1}\kappa_1' \in \text{SO}(p-1)$. Hence, $k=k' h$, for some $h \in H$ and $k\star \psi(z) = k'h \star \psi(z) = k' \star \psi(z)$

If $\alpha = -\alpha' \neq 0$, then $z= j_1(z')$ and $\kappa_1^{-1}\kappa_1' \in j_1 \text{SO}(p-1)$. Hence, $k = k'j_1h$ for some $h \in H$ and $k \star \psi(z) = k' j_1 h \star \psi \left( j_1(z') \right) = k'j_1j_1 \star \psi(z') = k' \star \psi(z')$

If $\alpha = \alpha' = 0$, then $z=z' = \left( \pm e_{p+1}, \epsilon_1 \right)$ and $\psi(z) = \zeta$, where $\zeta \in \Sigma$ is fixed by $\text{SO}(p)$, since $f_\Sigma(\zeta)=f(z)=[0:1]$. Moreover, $k=k' (g_1,g_2)$, with $(g_1,g_2) \in \text{SO}(p) \times \text{SO}(q-1)$. Then, $k \star \psi(z) = k' (g_1,g_2) \star \psi(z) = k' \psi(z)$
\end{proof}
\begin{lemma} \label{G-equiv map btw two axns}
    $F$ is $G$-equivariant.
\end{lemma}
\begin{proof}
    Let $g \in G$ and $x = k_0 \star z \in \text{S}^{p} \times \text{S}^{q-1}$, where $k_0 = (\kappa_1, \kappa_2) \in K$ and $z \in\mathcal{S}$. Write $gk_0 = k m(\theta) u$, for $k \in K$, $ \theta \in \mathbb{R}$ and $ u \in U(z)$, according to (\ref{eq:1}). Note that $U(z) = U(\psi(z))$. Then
    \begin{align*}
        F \left( g \star x) \right) & = F \left( k \star \Phi_\theta (z) \right)
         = k \star \psi \left( \Phi_\theta (z) \right)
         = k \star \Phi^\Sigma_\theta \left( \psi(z) \right)
         = km(\theta)u \star \psi(z) \\
        & = g k_0 \star \psi(z) = g \star \left( k_0 \star \psi(z) \right) = g \star F(z)
    \end{align*}
\end{proof}
\begin{lemma} \label{local anal isom brw two axns}
    $F$ is a local analytic isomorphism.
\end{lemma}
\begin{proof}
    By the $G$-equivariance and the fact that $G\star \mathcal{S}=\text{S}^p \times \text{S}^{q-1}$, it suffices to show that $F$ is a local analytic isomorphism around a $z \in \mathcal{S}$.

    $\bullet$ If $f(z) \neq \pm 1$, then $f_\Sigma (F(z)) \neq 1$ and so, the $G$-orbits of both $z$ and $F(z)$ are open. Then, $G$-equivariance of $F$ implies that $F$ is analytic at $z$ and (D$F)_z$ is an isomorphism and so we have the result. 
    
    $\bullet$ If $f(z)= \pm 1$, then the map $K \times \mathcal{S}  \rightarrow \text{S}^p \times \text{S}^{q-1}$ defined by $\left( k, z \right)  \mapsto k \star z$ is a submersion at $\left( (I_p, I_q) ,z \right)$, and so $F$ is analytic at $z$. Moreover, the dimension of the orbits of $z$ and $\psi(z)$ are $p+q-2$. Now, $G$-equivariance gives us that $(DF)_z$ is onto the tangent space of $\psi(z)$. On the other hand, the orbit of $\psi(z)$, $\mathcal{O}\left( \psi(z) \right)$, is locally isomorphic to $ G/ P $ where $P$ is a maximal parabolic for which $G/ P$ is diffeomorphic to $\text{S}^{p-1} \times \text{S}^{q-1}$. In this model, the fixed point set of $H$ is 0-dimensional. Therefore, in $M$, $\mathcal{F}$ and $\mathcal{O} \left( \psi(z) \right)$ intersect transversely. Hence, if $\gamma (t)$ is a non constant curve in $\mathcal{S}$ through $z$, we have $(DF)_z (\gamma ' (0) )  = (D \psi)_z (\gamma ' (0)) \neq 0$, since $\psi$ is an isomorphism, and therefore $(DF)_z( \gamma '(0) ) \in \text{T}_{\psi(z)}M \, \Big\backslash \,\text{T}_{\psi(z)}\mathcal{O} \left( \psi(z) \right)$. Hence, $(DF)_z$ is onto all of T$_{\psi(z)}M$, hence it is an isomorphism and we have the desired result.
\end{proof}
\begin{lemma}
    M = $F \left( \textup{S}^p \times \textup{S}^{q-1} \right)$
\end{lemma}
\begin{proof}
    Since $F$ is a local isomorphism, its image is open in $M$. Since S$^p \times$S$^{q-1}$ is compact, the image of $F$ is also closed. The result follows, since $M$ is connected.
\end{proof}

Finally, since $F$ is a local isomorphism that is onto $M$, and both $M$ and $\text{S}^p \times \text{S}^{q-1}$ are compact, it is a covering map. Hence, we have proved the first case of Theorem \ref{final result SOpq}. Of course, changing the roles of $\text{SO}(p)$ and $\text{SO}(q)$, in an analogous manner we get the second case as well.

\subsection{$\text{SO}^\circ(p,q)$ actions with both $\text{SO}(p)$ and $\text{SO}(q)$ fixed points} \label{axns with both fxd pts}

Recall $H= \text{SO}(p-1) \times \text{SO}(q-1)$ and $\mathcal{F} = \text{fixed point set of }H$. Assume that both $\text{SO}(p)$ and $\text{SO}(q)$ have fixed points in $M$. As it was noted in the beginning of Section \ref{SO(p,q) axns}, by Corollary \ref{fix pnt set of H nonempty}, $\mathcal{F}$ is non empty and all its connected components are 1-dimensional. Let $\Sigma$ be one of $\mathcal{F}$'s connected components. On $\Sigma$, by Lemma \ref{lemma Uch}, there exists an analytic function $f_\Sigma : \Sigma \rightarrow \mathbb{R}\text{P}^1$, see Remark \ref{rmk definition of f}. If $f \equiv [\pm 1:1]$, then it was shown in Section \ref{axns w nullcone} that $M$ is covered by $G \times_P \text{S}^1$, which contradicts the existence of $\text{SO}(p)$ and $\text{SO}(q)$ fixed points.
Then, by Lemma \ref{f =/ [1:1] lemma}, we may assume that there exists an $\text{SO}(p)$ fixed point on $\Sigma$.
Let $x \in M$ be a fixed point of $\text{SO}(q)$. First of all, we note that $x$ must be on $\Sigma$. Indeed, if it is not, the proof of the first case of Theorem \ref{final result SOpq}, see Section \ref{axn with SOp fixed pnt}, shows that $M$ is covered by $\text{S}^p \times \text{S}^{q-1}$ on which the action of $K$ is the standard one. Since $x$ is fixed by $\text{SO}(q)$, the fiber of $x$ would be a discrete, $\text{SO}(q)$-invariant subset of  
$\text{S}^p \times \text{S}^{q-1}$. However, evidently there are not any such sets. Being involutions on a circle with at least one fixed point, $j_1$ and $j_2$ have exactly two fixed points. Since they also commute, they either have the same fixed points on $\Sigma$ or their fixed point sets on $\Sigma$ are disjoint.

Consider $\text{S}^1 $ and let $\text{J}_1$ and $\text{J}_2$ be the reflections with respect to the $y$-axis and the $x$-axis respectively.

\noindent \underline{\textbf{Case 1:}} $j_1$ and $j_2$ have disjoint fixed point sets on $\Sigma$. \\
Similarly to Lemma \ref{circle isom 1}, we can show:
\begin{lemma} \label{circle isom 2}
 There exists an analytic isomorphism \[ \psi : \Sigma \rightarrow \text{S}^1 \] such that \[ \psi \circ J_i = j_i \circ \psi \]    
 for $i=1,2$.
\end{lemma}
Now, consider the sphere $\text{S}^{p+q-1} \subseteq \mathbb{R}^{p+q}$ and denote the standard basis of $\mathbb{R}^{p+q}$ by $\{ e_1 , \cdots e_{p+q} \}$. Let $\mathcal{S} = \{ \alpha e_1 + \beta e_{p+1} : \alpha^2 + \beta^2 =1 \} \subseteq \text{S}^{p+q-1}$. In the standard orthogonal action of $K$ on $\text{S}^{p+q-1}$, $j_1 = \begin{bmatrix}
    -I_2 & \\ & I_{p+q-2}
\end{bmatrix} \in \text{SO}(p)$ and $j_2 = \begin{bmatrix}
    I_{p+q-2} & \\ & -I_2
\end{bmatrix} \in \text{SO}(q)$ act by $j_1 \star \left( \alpha e_1 + \beta e_{p+1} \right) =  -\alpha e_1 + \beta e_{p+1}$ and $j_2 \star  \left( \alpha e_1 + \beta e_{p+1} \right) =  \alpha e_1 - \beta e_{p+1}$. Using Lemma \ref{circle isom 2}, we get an analytic isomorphism $\psi': \Sigma \rightarrow \text{S}^1$, such that $\psi' \circ j_i = \text{J}_i \circ \psi'$ for $i=1,2$. Let $\rho: \text{S}^1 \rightarrow \mathcal{S}$ be the analytic isomorphism defined by $\rho (\alpha e_1 + \beta e_{2})=\alpha e_1 + \beta e_{p+1}$ and let $\psi = \rho \circ \psi'$. Then, $\rho: \Sigma \rightarrow \mathcal{S} $ and $\psi \circ j_i = j_i \circ \psi$.

Now, similarly to the case of Section \ref{axn with SOp fixed pnt}, we get a flow and a function on $\mathcal{S}$, which satisfy the Uchida conditions, see Remark \ref{Uchida conditions}. Hence the flow is an induced basic $(j_1,j_2)$-flow. By \cite{Uchida}, we get an action of $\text{SO}^\circ(p,q)$ on $\text{S}^{p+q-1}$ from the Uchida construction corresponding to a basic $(\text{J}_1,\text{J}_2)$-flow, see Section \ref{Uchida constr section}. Then, essentially the same steps as in Section \ref{axn with SOp fixed pnt} show that this action covers equivariantly the action of $\text{SO}^\circ(p,q)$ on $M$.

\noindent \underline{\textbf{Case 2:}} $j_1$ and $j_2$ have the same fixed points on $\Sigma$.\\
First we note that $\displaystyle j_1 \Big|_\Sigma = j_2 \Big|_\Sigma$. Indeed, by considering a $j_1, j_2$ invariant metric on $\Sigma$, $j_1$ and $j_2$ have the same fixed points, which are isolated, and if $x$ is one of those fixed points, $\displaystyle \left(\text{D}j_1\right)_x = \left( \text{D}j_2 \right)_x = -\text{Id} $. Therefore, they are equal on $\Sigma$. On $\Sigma$ we also get an analytic flow $\Phi_\theta$ on $\Sigma$ induced by the action of the subgroup $\mathcal{M}(p,q)$. Because there are no nullcone orbits and because of Uchida condition (iii), see Remark \ref{Uchida conditions}, between two points fixed by $\Phi_\theta$ there must be a point fixed by $\text{SO}(p)$ or $\text{SO}(q)$ and vice versa. Therefore, on $\Sigma$ there are two points fixed by $\Phi_\theta$. Note that $\text{SO}(p)$ or $\text{SO}(q)$ cannot have a common fixed point, by Lemma \ref{K fixed pnt}.

A slight modification of the proof of Lemma \ref{circle isom 1} shows that there exists an analytic isomorphism $\psi : \Sigma \rightarrow  \mathbb{R}\text{P}^1$ such that if $j_0:\mathbb{R}\text{P}^1 \rightarrow \mathbb{R}\text{P}^1 $ is the map $[a:b] \mapsto [-a:b]$, then $\psi \circ j_i = j_0 \circ \psi$ for $i=1,2$. The fixed points of $j_0$ are $[1:0]$ and $[0:1]$. Moreover, via $\psi$, $\Phi_\theta$ induces a flow, $\Phi'_\theta$, on $\mathbb{R}\text{P}^1$ and let $\zeta_1$, $\zeta_2$ be the points fixed by $\Phi'_\theta$. Let $\pi : \text{S}^1 \rightarrow \mathbb{R}\text{P}^1$ be the standard covering map. Let $z_1, z_3 \in \text{S}^1$ be the points that map to $\zeta_1$ and $z_2,z_4 \in \text{S}^1$ the points that map to $\zeta_2$. Note that because of the property $\Phi_\theta (j_1 \star z) = j_1 \star \Phi_{-\theta}(z)$ for $\theta \in \mathbb{R}$ and $z \in \Sigma$, we have $\Phi'_\theta (j_0 \star \zeta) = j_0 \star \Phi'_{-\theta}(\zeta)$
for $\theta \in \mathbb{R}$ and $\zeta \in \mathbb{R}\text{P}^1$. Hence, if $\zeta_1 = [a:b]$ then $\zeta_2 = [-a:b]$. Therefore, we can assume that $z_2 = \text{J}_2(z_1), \, z_3 = \text{J}_1 \circ \text{J}_2 (z_1),$ and $z_4 = \text{J}_1(z_1)$. Note that $\text{J}_1 \circ \text{J}_2$ is the antipodal map on $\text{S}^1$.

We will now define a flow on $\text{S}^1$ that will cover the flow $\Phi'_\theta$. Let $\mathcal{S}_1$ be the connected component of $\text{S}^1 \backslash \{z_1,z_3\} $ that contains $z_2$. Then $\pi \Big|_{\mathcal{S}_1}: \mathcal{S}_1 \rightarrow \mathbb{R}\text{P}^1 \backslash \{\zeta_1\}$ is an analytic isomorphism. Define $\Phi''_\theta$ on $\mathcal{S}_1$ by
\[ \Phi''_\theta (z) = \left( \pi \Big|_{\mathcal{S}_1} \right)^{-1} \circ \Phi'_\theta \left( \pi(z) \right) \]
Then, on $\text{J}_1\circ \text{J}_2 \left( \mathcal{S}_1 \right)$, which is  the connected component of $\text{S}^1 \backslash \{z_1,z_3\} $ that contains $z_4$, define $\Phi''_\theta$ by $\Phi''_\theta (z) = (\text{J}_1\circ \text{J}_2) \circ \Phi''_\theta \left( \text{J}_1\circ \text{J}_2(z) \right)$. Note that, on $\{ z_1,z_2,z_3,z_4 \}$ $\Phi''_\theta$ is constant. It is immediate that $\Phi''_\theta$ is an analytic flow on $\text{S}^1$ such that $\pi \circ \Phi''_\theta (z) = \Phi'_\theta \circ \pi (z)$ for $\theta \in \mathbb{R}$ and $z \in \text{S}^1$. Now, define a function 
\[ f_{\text{S}^1} := f_\Sigma \circ \pi\]
It is a straightforward calculation that $\Phi''_\theta$ and $f_{\text{S}^1}$ satisfy the Uchida conditions, see Remark \ref{Uchida conditions}.

 As in Case 1 above, there exists an analytic action of $\text{SO}^\circ(p,q)$ on $\text{S}^{p+q-1}$ from the Uchida construction corresponding to a basic $(\text{J}_1,\text{J}_2)$-flow, see Section \ref{Uchida constr section}, that covers the action of $\text{SO}^\circ(p,q)$ on $M$. Hence, we have proved the third case of Theorem \ref{final result SOpq}.
\appendix
\section{Analyticity of the actions in the basic construction} \label{analyticity of actions}
Following a similar approach as in \cite{Uchida}, we will show analyticity of the action defined in (\ref{axn def}) by writing $G \times \left( \text{S}^p \times \text{S}^{q-1} \right)$  as a union of three open sets and showing that the action map is analytic when restricted to any of them. Two of these open sets are going to be $G$ times the $G$-orbit of the two fixed points of $\text{SO}(p)$, namely $(\pm e_{p+1}, \epsilon_1)$, while the third one will be an open set that accounts for the closed $G$-orbits in $\text{S}^p \times \text{S}^{q-1}$.

Let $\mathcal{O}$ be the orbit of $(e_{p+1},\epsilon_1)$ under the action defined in (\ref{axn def}).

\begin{prop} \label{analytic axn orbit prop}
The $G$ action defined in (\ref{axn def}) is analytic on $\mathcal{O}$.
\end{prop}

 In order to prove the proposition, we will define a $G$-equivariant, analytic isomorphism from $\mathcal{O}$ to $\mathcal{O}_1$, where $\mathcal{O}_1$ is the orbit of $e_{p+1}$ under the standard projective $G$ action on S$^{p+q-1}$, which we see as a subset of $\mathbb{R}^{p+q}$ with its standard basis; we will however rename the standard basis as $\{ e_1, \cdots e_p, \epsilon_1, \cdots , \epsilon_q \}$. We will break the proof of the proposition in a series of lemmas.

For $x\in \mathbb{R}^p$ and $y \in \mathbb{R}^q$, we write $x \oplus y$ for the vector $\begin{bmatrix}
    x\\y
\end{bmatrix} \in \mathbb{R}^{p+q}$. Note that
\begin{equation*}
    \mathcal{O}^1 = \{ x \oplus y \in \text{S}(\mathbb{R}^p \oplus
     \mathbb{R}^q) : \lVert x \rVert < \lVert y \rVert \}
\end{equation*}
Recall $\mathcal{S} =  \big\{ ( \alpha e_1 + \beta e_{p+1} , \epsilon_1 ) : \alpha^2 + \beta^2 =1 \big\} \subseteq \text{S}^p \times \text{S}^{q-1}$. Let $\mathcal{I}$ be the intersection of $\mathcal{S}$ and $\mathcal{O}$, namely
\begin{equation*}
    \mathcal{I}  = \big\{ \Phi_\theta ( e_{p+1}, \epsilon_1 ) : \theta \in \mathbb{R} \big\}
\end{equation*}
By property (A3) of Remark \ref{(A1)-(A4) props remark}, $f$ is an analytic isomorphism between $\mathcal{I}$ and $(-1,1)$. Recall that $f(e_{p+1},\epsilon_1)=0$. We can assume that 
\begin{equation} \label{assumption f>0}
f > 0 \text{ on } \{ z \in \mathcal{I} : z=(\alpha_z e_1 + \beta_z e_{p+1},\epsilon_1) \text{ with } \alpha_z > 0 \}
\end{equation}
For $(v,w) \in \mathcal{O}$, we write $v$ as $v= \begin{bmatrix}
    v_0 \\ v_{p+1}
\end{bmatrix}$, with $v_0 \in \mathbb{R}^p$ and $v_{p+1} \in \mathbb{R}$. Define a function $F : \mathcal{O}  \rightarrow B^{p}_1(0) \times \text{S}^{q-1}$ by
\begin{equation*}
    F(v,w)  = \begin{cases}
      \frac{f\big( \lVert v_0 \rVert e_1 + v_{p+1}e_{p+1},\epsilon_1 \big)}{\lVert v_0 \rVert} v_0 \oplus w  \quad & \text{if } v \neq e_{p+1} \\
      0 \oplus w & \text{if }v = e_{p+1}
    \end{cases} 
\end{equation*}
where $B^p_1(0)$ is the unit open ball in $\mathbb{R}^p$ centered at 0. Define $\phi : \mathcal{O} \rightarrow \mathcal{I}$ by $(v,w) \mapsto ( \lVert v_0 \rVert e_1 + v_{p+1}e_{p+1},\epsilon_1)$, so that $\displaystyle F(v,w) = \Big( \frac{f\big( \phi(v,w) \big)}{\lVert v_0 \rVert} v_0 \oplus w \Big)$. 
\begin{lemma} \label{analytic axn on orbit lemma 1}
    The function $F : \mathcal{O}  \rightarrow B^{p}_1(0) \times \text{S}^{q-1}$ is analytic.
\end{lemma}
\begin{proof}
It is easy to see that $F$ analytic at any point with $v \neq e_{p+1}$, since $f$ is analytic.  Around $(e_{p+1},\epsilon_1)$, we parametrise $\mathcal{S}$ by 
\begin{align} \label{chart psi}
    \psi : & ( -\delta , \delta )  \rightarrow \mathcal{V} \\
    & s  \mapsto ( se_1 + \sqrt{1-s^2} e_{p+1}, \epsilon_1 )
\end{align}
for some $0< \delta < 1$ and $\mathcal{V} \subseteq \mathcal{S}$ a neighbourhood of $(e_{p+1}, \epsilon_1)$. Set 
\begin{equation} \label{f tilde}
    \tilde{f}(s) := f \circ \psi =f \big( se_1 + \sqrt{1-s^2}e_{p+1} , \epsilon_1\big)
\end{equation}
Then, $\tilde{f}$ is an analytic function on $(- \delta , \delta)$. Moreover, $\tilde{f}(0) = 0$ and \begin{align*}f \big( -se_1 + \sqrt{1-s^2}e_{p+1}, \epsilon_1 \big) & = f \big( j_1 ( se_1 + \sqrt{1-s^2}e_{p+1}, \epsilon_1) \big) \\ 
& = - f \big( se_1 + \sqrt{1-s^2}e_{p+1} , \epsilon_1 \big)
\end{align*}
by (A2). Hence, $\tilde{f}(-s)=-\tilde{f}(s)$. Consider the function on $(-\delta,\delta)$ defined by
\begin{equation*}
 \mathcal{H}(s) =
\begin{cases}
    \frac{\tilde{f}(s)}{s} \, , s\neq 0\\
    \tilde{f}'(0) \, , s=0
\end{cases}
\end{equation*}
Since $\tilde{f}$ is analytic and $\tilde{f}(0)=0$, the function $\mathcal{H}$ is analytic, and it is an even function. Smoothness of $\mathcal{H}$ suffices to give us the existence of a smooth function $\tilde{\mathcal{H}}$, defined around 0 such that $\tilde{\mathcal{H}}(s^2) = \mathcal{H}(s)$, see \cite[Ch VIII, \S 14 , Problem 6]{dieudonnefoundations}. In fact, $\tilde{\mathcal{H}}$ can actually be taken analytic. Indeed, since $\mathcal{H}$ is analytic, there exists a power series centered at 0 that converges to $\mathcal{H}(s)$ for $s$ close to 0. However, $\mathcal{H}$ is an even function, which implies that all its derivatives of the form $\mathcal{H}^{(2n+1)}$ for $n \in \mathbb{N}$, are odd functions and so $\mathcal{H}^{(2n+1)}(0)=0$. Therefore, all the odd powers in the Taylor series of $\mathcal{H}$ around 0 vanish. Hence, there are real numbers $\{ c_2, c_4, 
\cdots , c_{2n}, \cdots \}$ and $0 < \delta ' < \delta$ such that for any $\lvert s \rvert < \delta '$, the series $ \sum_{n \geq 1}c_{2n}s^{2n}$ converges to $\mathcal{H}(s)$. Now, consider the power series $\sum_{n \geq 1}c_{2n}s^n$. Then $\limsup_{n \rightarrow \infty} ( \sqrt[n]{\lvert c_{2n} \rvert} ) = \limsup_{n \rightarrow \infty} \big( [\sqrt[2n]{\lvert c_{2n} \rvert}]^2 \big) = \Big( \limsup_{n \rightarrow \infty} ( \sqrt[2n]{\lvert c_{2n} \rvert} ) \Big)^2$ which converges, since $\frac{1}{\limsup_{n \rightarrow \infty} ( \sqrt[2n]{\lvert c_{2n} \rvert} )}$ is the radius of convergence of $\sum_{n \geq 1}c_{2n}s^{2n}$. Hence, the series $\sum_{n \geq 1}c_{2n}s^n$ converges and defines an analytic function around 0. Then, we can take $\tilde{\mathcal{H}}$ to be defined by this series and by the uniqueness of Taylor series, we see that $\tilde{\mathcal{H}}(s^2) = \mathcal{H}(s)$. \\
Now, the function $v \mapsto \tilde{\mathcal{H}}(\lVert v_0 \rVert^2 ) $ is an analytic function around 0. Hence, the function $v \mapsto \mathcal{H}(\lVert v_0 \rVert )$ is an analytic function around 0. But $\mathcal{H}(\lVert v_0 \rVert ) =  \frac{f\big( \phi(v,w) \big)}{\lVert v_0 \rVert}$. Consequently, the function $v \mapsto \frac{f\big( \phi(v,w) \big)}{\lVert v_0 \rVert} \, v_0$ is analytic around $v= e_{p+1}$, which in turn implies that the function $F$ is analytic around $(e_{p+1}, \epsilon_1)$.
\end{proof}
Subsequently, we define $F_0 : \mathcal{O} \rightarrow \mathcal{O}^1$ by
\begin{equation*}
    F_0(v,w) = \frac{1}{\sqrt{1 + f^2 \big( \phi(v,w) \big)}} F(v,w)
\end{equation*}
\begin{lemma} \label{analytic axn on orbit lemma 2}
    The function $F_0: \mathcal{O} \rightarrow \mathcal{O}^1$ is a $G$-equivariant, analytic isomorphism.
\end{lemma}
\begin{proof}
By Lemma \ref{analytic axn on orbit lemma 1}, $F_0$ is analytic. Before we calculate its inverse, we show that $F_0$ is $G$-equivariant. Firstly, we show equivarience with respect to $K$. Let $k =(\kappa_1, \kappa_2) \in K$ and $(v,w) \in \mathcal{O}$. Recall $\displaystyle \tilde{\kappa} = 
    \begin{bmatrix}
        \kappa & \\ & 1
    \end{bmatrix}$ for $\kappa \in \text{SO}(p)$; see (\ref{SOp embed in SOp+1}). Then, $k\star (v,w) = (\kappa_1, \kappa_2) \star (v,w)  = ( \tilde{\kappa}_1 v , \kappa_2 w) 
   = \big( \begin{bmatrix}
        \kappa_1 v_0 \\ 
        v_{p+1} 
    \end{bmatrix}, \kappa_2 w \big)$. Now, $\phi \big( \begin{bmatrix}
        \kappa_1 v_0 \\ 
        v_{p+1} 
    \end{bmatrix}, \kappa_2 w \big) =  ( \lVert \kappa_1 v_0 \rVert e_1 + v_{p+1}e_{p+1},\epsilon_1) = ( \lVert v_0 \rVert e_1 + v_{p+1}e_{p+1},\epsilon_1) = \phi(v,w)$
since $\kappa_1 \in\text{SO}(p)$. Therefore,
\begin{align*}
    F_0 \Big( (\kappa_1,\kappa_2) \star (v,w) \Big) & = \frac{1}{\sqrt{1 + f^2 \big( \phi\big( (\kappa_1,\kappa_2) \star (v,w) \big) \big)}} ( \kappa_1 v_0 \oplus \kappa_2 w) \\
    & = \frac{1}{\sqrt{1 + f^2 \big( \phi (v,w) \big)}} ( \kappa_1 v_0 \oplus \kappa_2 w) = (\kappa_1,\kappa_2) \star F_0(v,w)
\end{align*}
where the action in the right hand side of the last equality is the orthogonal action of $K$ on $\text{S}^{p+q-1}$. Therefore, we have $K$-equivarience. Now, we let $\theta \in \mathbb{R}$ and write $m(\theta) \star (e_{p+1},\epsilon_1) = \Phi_\theta (e_{p+1},\epsilon_1) = ( \alpha_\theta e_1+ \beta_\theta e_{p+1}, \epsilon_1 )$. Then,
\begin{equation*} 
    F_0 \big( \Phi_\theta (e_{p+1},\epsilon_1 ) \big) = \frac{1}{\sqrt{1+\text{tanh}^2(|\theta |)}} \Big( \frac{\text{tanh} ( \lvert \theta \rvert ) }{\lvert \alpha_\theta \rvert} \alpha_\theta e_1 \oplus \epsilon_1 \Big) \\
\end{equation*}
Recall that $f \left( \Phi_\theta (e_{p+1},\epsilon_1) \right) = \tanh{(\theta)}$ by relation (A3) in Remark \ref{(A1)-(A4) props remark} in Section \ref{actions on Spq}. If $\theta > 0$, we have that $f \left( \Phi_\theta (e_{p+1},\epsilon_1) \right) = \tanh{(\theta)} >0$. Then $ \alpha_\theta > 0$ by assumption (\ref{assumption f>0}) and so $\frac{\text{tanh} ( \lvert \theta \rvert ) }{\lvert \alpha_\theta \rvert} \alpha_\theta e_1 = \frac{\text{tanh}(\theta)}{\alpha_\theta} \alpha_\theta e_1 = \text{tanh}(\theta ) e_1$. On the other hand, if $\theta < 0$, then $\alpha_\theta < 0$ and then also $\frac{\text{tanh} ( \lvert \theta \rvert ) }{\lvert \alpha_\theta \rvert} \alpha_\theta e_1 = \frac{\text{tanh}(- \theta)}{- \alpha_\theta} \alpha_\theta e_1 = \text{tanh}(\theta ) e_1$. Hence
\begin{equation} \label{F0 img of m.e1}
    F_0 \big( \Phi_\theta (e_{p+1},\epsilon_1 ) \big) = \frac{1}{\sqrt{1+\text{tanh}^2(|\theta |)}} \Big( \tanh{(\theta)} e_1 \oplus \epsilon_1 \Big)
\end{equation}
Computing now the action of $m(\theta)$ on $F_0(e_{p+1},\epsilon_1) = \epsilon_1$, we get
\begin{align*}
    \begin{bmatrix}
        \text{cosh}(\theta) && \text{sinh}(\theta) & \\
        & I_{p-1} & & \\
        \text{sinh}(\theta) && \text{cosh}n(\theta) & \\
        &&& I_{q-1}
    \end{bmatrix}
    \epsilon_1 & = \begin{bmatrix}
        \text{cosh}(\theta) && \text{sinh}(\theta) & \\
        & I_{p-1} & & \\
        \text{sinh}(\theta) && \text{cosh}(\theta) & \\
        &&& I_{q-1}
    \end{bmatrix} 
    \begin{bmatrix}
        \\
        \\
        1\\
        \\
    \end{bmatrix} 
    \\
    & = \begin{bmatrix}
        \text{sinh}(\theta) \\
        \\
        \text{cosh}(\theta) \\
        \\
    \end{bmatrix}
     = \text{sinh}(\theta) e_1 \oplus \text{cosh}(\theta) \epsilon_1 \\
\end{align*}
The Euclidean norm of $\text{sinh}(\theta) e_1 \oplus \text{cosh}(\theta) \epsilon_1$ is $\text{cosh}(\theta) \sqrt{1+ \text{tanh}^2(\theta)}$. Hence,
\begin{align*}
    m(\theta) \star F_0(e_{p+1},\epsilon_1 ) & =
    m(\theta) \star \epsilon_1 = \frac{1}{\text{cosh}(\theta) \sqrt{1+ \text{tanh}^2(\theta)}} \big( \text{sinh}(\theta) e_1 \oplus \text{cosh}(\theta) \epsilon_1 \big)\\
    & = \frac{1}{\sqrt{1+ \text{tanh}^2(\theta)}} \big( \text{tanh}(\theta) e_1 \oplus \epsilon_1 \big)  = F_0 \left( \Phi_\theta(e_{p+1}\epsilon_1) \right) \text{   (by (\ref{F0 img of m.e1}))}\\
    & = F_0 \big( m(\theta) \star (e_{p+1}, \epsilon_1) \big)
\end{align*}
where the action in either side of the first equality is the projective action of $G$ on $\text{S}^{p+q-1}$. Now, we can show $G$-equivarience. Let $g \in G$ and, according to (\ref{eq:1}), write $g = k m(\theta) u$, with $u \in U(e_{p+1},\epsilon_1) = H^\circ_{[0:1]}$, see (\ref{U(z) eq definition}). Note that then, $u$ fixes $F_0(e_{p+1},\epsilon_1)$. Then
\begin{align*}
    F_0 (g \star (e_{p+1},\epsilon_1)) & = F_0( km(\theta) \star (e_{p+1}, \epsilon_1)) = \left( km(\theta) \right) \star F_0(e_{p+1}, \epsilon_1) \\
    & = \left( km(\theta)  \right)\star \epsilon_1 = \left( km(\theta) u \right) \star \epsilon_1 = g \star F_0(e_{p+1},\epsilon_1)
\end{align*}

Finally, we define the inverse of $F_0$. Let $x \oplus y \in \mathcal{O}^1 \\ = \left\{ x \oplus y \in \text{S}^{p+q-1}:\lVert x\rVert < \lVert y\rVert \right\}$. Recall that $f$ restricted on \\ $\mathcal{I} = \big\{ \Phi_\theta ( e_{p+1}, \epsilon_1 ) : \theta \in \mathbb{R} \big\}$ is invertible. Let $\text{pr}_{\text{S}^p} : \text{S}^p \times \text{S}^{q-1} \rightarrow \text{S}^p$ be the projection to the first factor. Let $\hat{f}$ be the function
\begin{equation} \label{f hat}
    \hat{f} = \text{pr}_{\text{S}^p} \circ \left(f \big|_\mathcal{I} \right)^{-1}
\end{equation}
Namely, if $f \left( \alpha \, e_1 + \beta \, e_{p+1} , \epsilon_1 \right) = s$ then, $\hat{f}(s) = \alpha \, e_1 + \beta \, e_{p+1}$. Define
\begin{equation} \label{v and w def}
    v(x\oplus y) = \begin{bmatrix}
        \frac{ \langle \hat{f}\big( \frac{\lVert x \rVert}{\lVert y \rVert } \big), e_1 \rangle}{\lVert x \rVert} x \\
        \langle \hat{f}\big( \frac{\lVert x \rVert}{\lVert y \rVert } \big), e_{p+1} \rangle
    \end{bmatrix}
\text{  and  }
    w(x\oplus y) = \frac{1}{\lVert y \rVert} y
\end{equation}
where $\langle \cdot , \cdot \rangle$ is the euclidean inner product in $\mathbb{R}^{p+1}$. 
Note that, for $x \oplus y \in \mathcal{O}^1$ we have $\lVert x \rVert < \lVert y \rVert$ and hence, $y\neq0$. Therefore, $w$ is well defined and analytic on $\mathcal{O}^1$. Define
\begin{align*}
    F_1 :& \mathcal{O}^1  \rightarrow \mathcal{O} \\
    &  x \oplus y  \mapsto \big( v(x\oplus y), w( x \oplus y ) \big) 
\end{align*}
Suppose $x \neq 0$. Then, $\left(f \big|_\mathcal{I} \right)^{-1}\big( \frac{\lVert x \rVert}{\lVert y \rVert } \big) = \big( \alpha e_1 + \beta e_{p+1},\epsilon_1 \big)$ with $\alpha > 0$. Let $F_1(x \oplus y) = (v,w)$. We have $v_0 = \frac{\alpha}{\lVert x \rVert}x$, $v_{p+1}= \beta$ and $w = \frac{1}{\Vert y \rVert}y$. Then, $f \big( \phi(v,w) \big) = \frac{\lVert x \rVert }{\lVert y \rVert}$. As a result,
\begin{align*}
    F_0(v,w) & = \frac{1}{\sqrt{1+\frac{\lVert x \rVert^2}{\lVert y \rVert^2}}} \left( \frac{\frac{\lVert x \rVert}{\lVert y \rVert}}{\alpha} \frac{\alpha}{\lVert x \rVert} x \oplus \frac{1}{ \lVert y \rVert} y \right) \\
    & = \frac{\lVert y \rVert}{\sqrt{\lVert y \rVert^2 + \lVert x \rVert^2}} \Big( \frac{1}{\lVert y \rVert} x \oplus \frac{1}{\lVert y \rVert} y \Big) = \frac{1}{\sqrt{\lVert y \rVert^2 + \lVert x \rVert^2}} \big( x \oplus y \big) 
 \end{align*}
But, $x\oplus y \in \mathcal{O}^1 \subset \text{S}^{p+q-1}$ and so $\lVert y \rVert^2 + \lVert x \rVert^2 = 1$. Hence, 
\begin{equation} \label{F0 F1 inverse eq 1}
    F_0(v,w) = F_0 \big( F_1(x \oplus y) \big) = x \oplus y
\end{equation}

If $x=0$, then $y \in \text{S}^{q-1}$. Hence, $v(0\oplus y) = e_{p+1}$ and $w(0 \oplus y) = y$, by (\ref{v and w def}). Then, $F_1(0 \oplus y) = \left( e_{p+1}, y \right)$. Since $f(e_{p+1}, \epsilon_1)=0$, by the definition of $F$ and $F_0$, we have $F_0(e_{p+1},y) = 0 \oplus y$. Therefore, 
\begin{equation} \label{F0 F1 inverse eq 2}
    F_0 \left( F_1(0 \oplus y) \right) = 0 \oplus y
\end{equation}
As a result, by equations (\ref{F0 F1 inverse eq 1}) and (\ref{F0 F1 inverse eq 2}) we have that $F_0 \big( F_1(x \oplus y) \big) = x \oplus y$ for any $x \oplus y \in \mathcal{O}^1$. Similarly, we can see that $F_1 \big( F_0(v,w) \big) = (v,w)$ and so $F_1 = F_0^{-1}$.

As for the analyticity of $F_0^{-1}$, it follows easily at points with $x \neq 0$. In order to deal with the case $x=0$, set $\displaystyle v_0 = \frac{ \langle \hat{f}\big( \frac{\lVert x \rVert}{\lVert y \rVert } \big), e_1 \rangle}{\lVert x \rVert} x$ and $v_{p+1} = \langle \hat{f}\big( \frac{\lVert x \rVert}{\lVert y \rVert } \big), e_{p+1} \rangle$. Now, around $(e_{p+1}, \epsilon_1)$ in $\mathcal{S}$, we use the chart given by the inverse of $\psi$ from (\ref{chart psi}), namely $\psi^{-1} :  \mathcal{V}   \rightarrow ( - \delta, \delta )$ such that $ \psi^{-1} \big( s e_1 + \sqrt{1-s^2} e_{p+1} , \epsilon_1 \big)  \mapsto s$.
Consider the function $h' = \psi^{-1} \circ \left(f \big|_\mathcal{I} \right)^{-1}$
We note that $h' = \tilde{f}^{-1}$, see (\ref{f tilde}). Since $\tilde{f}$ is an odd function, $h'$ is also an odd analytic function. In addition, $h'(0)=0$. Then, the function $s \mapsto \frac{h'(s)}{s}$ is an even, analytic function on $(-\delta,\delta)$. Using the same argument as in the proof of Lemma \ref{analytic axn on orbit lemma 1}, we conclude that the function $ \displaystyle x \oplus y \mapsto \frac{\left(f \big|_\mathcal{I} \right)^{-1}\Big( \frac{\lVert x \rVert}{\lVert y \rVert} \Big)}{\frac{\lVert x \rVert}{\lVert y \rVert}}$ is an analytic function around $0 \oplus y \in \mathcal{O}^1$. Therefore, $ \displaystyle x \oplus y \mapsto \frac{\left(f \big|_\mathcal{I} \right)^{-1}\Big( \frac{\lVert x \rVert}{\lVert y \rVert} \Big)}{\lVert x \rVert}$ is an analytic function around $0 \oplus y \in \mathcal{O}^1$. Hence $v_0$ is an analytic function. As for $v_{p+1}$, around $x=0$, it is equal to the function $\displaystyle v_0 \mapsto \sqrt{1 - v_0^2}$. However, $v_0 \neq 1$ around $x=0$ and we just saw that it is analytic. Hence, $v_{p+1}$ is also analytic.
\end{proof}

Proposition \ref{analytic axn orbit prop} now follows from Lemma \ref{analytic axn on orbit lemma 2}. Similarly, it can be shown that the $G$ action restricted to the orbit of $(-e_{p+1}, \epsilon_1)$ is analytic.

Now, let $\mathcal{S}_+ = \{ z \in \mathcal{S} : f(z) > 0 \} = \{ z = ( \alpha_z e_1 + \beta_z e_{p+1} , \epsilon_1) : \alpha_z > 0 \}$
Also, define $D_+ = \left\{ (\theta, z) \in \mathbb{R}  \times \text{S}^1 : \Phi_\theta (z) \in \mathcal{S}_+ \right\}$ and $W_+= \left\{ (g,z) \in G\times \mathcal{S}_+ : \pm Tr \big( gP(z)g^T \big) \neq \frac{1-f^2(z)}{1+f^2(z)} \right\}$ as in \cite{Uchida}. See equation (\ref{P(z) eq}) for the definition of the matrices $ P(z)$. We then have

\begin{lemma}{\cite[Lemma 4.7]{Uchida}} \label{analytic axn lemma 3}
    For any $(g,z) \in W_+$, there exist unique $kH \in K/H$ and $\theta \in \mathbb{R}$ such that \begin{equation} \label{eq:5}
    (\theta , z) \in D_+ \text{  and   } g=km(\theta)u
    \end{equation}
for some $u \in U(z)$. Moreover, the function \begin{align*}
    \Delta : W_+ & \rightarrow K/H \times D_+ \\
    (g,z) & \mapsto \left( kH , \theta,z \right)
\end{align*} is analytic.
\end{lemma}
The proof is similar to \cite{Uchida}. See also \cite{Lentas_thesis} for more details. Now, define the set $W_0 := \left\{ ( g , k \star z) : (gk,z) \in W_+ \right\}$. It is easy to see that $W_0$ is an open subset of $G \times \left( \text{S}^p \times \text{S}^{q-1} \right)$.
\begin{prop}
    The action map of the $G$ action defined in (\ref{axn def}) is analytic on $W_0$.
\end{prop}
\begin{proof}
The action map restricted to $W_0$ is $g \star \left( k \star z \right) = \Delta_1(gk,z) \star \Phi_{\Delta_2(gk,z)} (z)$ and hence the action is analytic, when restricted to $W_0$, by Lemma \ref{analytic axn lemma 3}.
\end{proof}

 Finally, we only have to observe that that 
\[ G \times \big( \text{S}^p \times \text{S}^{q-1} \big) =( G \times \mathcal{O} ) \cup ( G \times \tilde{\mathcal{O}}) \cup W_0 \]
where $\mathcal{O}$ and $\tilde{\mathcal{O}}$ are the $G$-orbits of $(e_{p+1}, \epsilon_1)$ and $(-e_{p+1}, \epsilon_1)$ respectively. Each of the sets on the right hand side is open, and the action map is analytic when restricted to either of them. Hence, the action of $G$ on $\text{S}^p \times \text{S}^{q-1}$ defined by (\ref{axn def}) is analytic.
\printbibliography
\end{document}